\documentclass[10pt]{amsart}

\usepackage{graphicx,subfig}
\graphicspath{{./pics/}}
\usepackage{bm,color}
\usepackage[section]{algorithm}
\usepackage{algorithmic}
\usepackage{enumerate}
\usepackage{verbatim}

\usepackage[colorlinks,linktocpage,linkcolor=blue]{hyperref}

\numberwithin{equation}{section}

\def\floor(#1){\left \lfloor #1 \right \rfloor}

\def\beq#1\eeq{\begin{equation} #1 \end{equation}}
\def\bal#1\eal{\begin{aligned} #1 \end{aligned}}
\usepackage{color}
\definecolor{black}{rgb}{0,0,0}

\definecolor{red}{rgb}{1,0,0}

\definecolor{blue}{rgb}{0,0,1}

\theoremstyle{plain}
\newtheorem{theorem}{Theorem}[section]
\newtheorem{lemma}[theorem]{Lemma}

\newtheorem{corollary}[theorem]{Corollary}

\newtheorem{assumption}{Assumption}[section]

\newtheorem{remark}{Remark}[section]

\graphicspath{{figures/}}

\def\cF{{\mathcal F }}

\def\tu{\tilde u}
\def\tf{\widetilde f}

\def\al{{\alpha}}
\def\be{{\beta}}

\def\RR{{\mathbb R}}
\def\II{{(0,1)}}

\begin{document}

\title[Fractional order variational problems]
{Variational formulation of problems involving fractional order
differential operators}
\author {Bangti Jin\and Raytcho Lazarov \and Joseph Pasciak}
\address {Department of Mathematics and Institute for
Applied Mathematics and Computational Science, Texas A\&M University,
College Station, TX 77843-3368 ({\texttt{btjin,lazarov,pasciak@math.tamu.edu}})}
\date{started July, 2012; today is \today}

\begin{abstract}
In this work, we consider boundary value problems involving Caputo and Riemann-Liouville
fractional derivatives of order $\alpha\in(1,2)$ on the unit interval $(0,1)$.
These fractional derivatives lead to non-symmetric boundary value problems,
which are investigated from a variational point of view.
The variational problem for the Riemann-Liouville case is coercive on the space
$H_0^{\alpha/2}(0,1)$ but the solutions are less regular, whereas that for
the Caputo case involves different test and trial spaces. The numerical analysis of these
problems requires the so-called shift theorems which show that the solutions of the variational
problem are more regular. The regularity pickup enables one to establish
convergence rates of the finite element approximations.
Finally, numerical results are presented to illustrate the error estimates.
\end{abstract}

\maketitle

\newcount\icount

\def\DD#1#2{\icount=#1
  \ifnum\icount<1
  \,_{ 0}\kern -.1em D^{#2}_{\kern -.1em x}
  \else
  \,_{x}\kern -.2em D^{#2}_1
  \fi
}

\def\DDRI#1#2{\icount=#1
  \ifnum\icount<1
  \,_{-\infty}^{\kern 1em R}\kern -.2em D^{#2}_{\kern -.1em x}
  \else
  \,_{x}^R \kern -.2em D^{#2}_\infty
  \fi
}

\def\DDR#1#2{\icount=#1
  \ifnum\icount<1
 _{0}^{ \kern -.1em R} \kern -.2em D^{#2}_{\kern -.1em x}
  \else
 _{x}^{ \kern -.1em R} \kern -.2em D^{#2}_{\kern -.1em 1}
  \fi
}

\def\DDCI#1#2{\icount=#1
  \ifnum\icount<1
  \,_{-\infty}^{\kern 1em C}  \kern -.2em D^{#2}_{\kern -.1em x}
  \else
  \,_{x}^C \kern -.2em  D^{#2}_\infty
  \fi
}

\def\DDC#1#2{\icount=#1
  \ifnum\icount<1
  \,_{0}^C \kern -.2em  D^{#2}_{\kern -.1em x}
  \else
  \,_{x}^C \kern -.2em D^{#2}_1
  \fi
}

\def\Hd#1{\widetilde H^{#1}}

\def\Hdi#1#2{\icount=#1
  \ifnum\icount<1
  \widetilde H_{L}^{#2}\II
  \else
  \widetilde H_{R}^{#2}\II
  \fi
}

\def\Cd#1{\icount=#1
  \ifnum\icount<1
  \widetilde C_{L}
  \else
  \widetilde C_{R}
  \fi
}

\section{Introduction}
In this paper we consider the following fractional-order source problem:
find $u$ such that
\begin{equation}\label{strongp}
  \begin{aligned}
   -\DD 0 \alpha u(x) + qu = f, & \quad x \in \II\\
   u(0)=u(1)=0, &
   \end{aligned}
\end{equation}
where $\alpha\in(1,2)$ is the order of the derivative, and $\DD 0\al$ refers to either
the Caputo or Riemann-Liouville fractional derivative of order $\alpha$ defined
below by \eqref{Caputo} and \eqref{Riemann}, respectively. Here $f$ is a function
in $L^2\II$ or other suitable Sobolev space. The potential coefficient
$q \in L^\infty(0,1)$ is a bounded measurable function.

The interest in the model \eqref{strongp} is motivated by the studies on anomalous diffusion
processes. Diffusion is one of the most prominent and ubiquitous transport mechanisms
found in nature. At a microscopic level, it is the result of the random motion of individual
particles, and the use of the Laplace operator in the canonical diffusion model rests on a
Brownian assumption on the particle motion. This microscopic explanation was established by
physicist Albert Einstein at the beginning of 20th century. However, over the last two decades,
a large body of literature has shown that
anomalous diffusion, in which the mean square variances grows faster (superdiffusion) or slower
(subdiffusion) than that in a Gaussian process, offers a superior fit to experimental data observed in a number
of important practical applications, e.g., viscoelastic materials, soil contamination and underground
water flow. The model \eqref{strongp} represents the steady state of one-dimensional superdiffusion process.
It can be derived from the following observation at a microscopic level: the particle motion might be dependent,
and can frequently take very large steps, following some heavy-tailed probability distribution.
The long range correlation causes the underlying stochastic process to deviate significantly
from the Brownian motion for the classical diffusion process. The macroscopic counterpart
is spatial fractional diffusion equations (SpFDEs), and we refer to \cite{BensonWheatcraftMeerschaert:2000}
for the derivation and relevant physical explanations. Numerous experimental studies have
convincingly demonstrated that SpFDEs can provide faithful description of the superdiffusion
process, and thus have attracted considerable attentions in practical applications.

Because of their extraordinary modeling capability, the accurate simulation of SpFDEs has become
an important task. Like in the classical diffusion case, closed form solutions are available only in a
few very limited cases, and generally one has to resort to numerical methods for accurate simulations.
A number of numerical methods, prominently the finite difference method, have been developed for the
time-dependent superdiffusion problem in the literature. In \cite{TadjeranMeerschaertScheffler:2006},
a second-order finite difference scheme based on Crank-Nicolson scheme in time and the shifted Gr\"{u}nwald
formula in space is proposed for one-dimensional diffusion problem with a Riemann-Liouville derivative
in space, and the stability, consistency and convergence of the scheme are provided. See
\cite{TadjeranMeerschaert:2007} for the two-dimensional extension and also
\cite{PodlubnyChechkin:2009,Sousa:2009} for related works. In \cite{Shkhanukov:1996}, Shkhanukov analyzed
a finite difference scheme for a second-order differential equation with a fractional derivative in the lower-order
term. Recently, high-order finite difference schemes like the weighted and shifted Gr\"{u}nwald difference
have received some interest \cite{ZhouTianDeng:2013}.
We note that the Gr\"{u}nwald formula on a uniform stencil leads to Toeplitz type matrices which can
be exploited for space storage reduction and fast solution via fast Fourier transform \cite{WangBasu:2012}
and for designing efficient multigrid preconditioners \cite{PangSun:2012}.

In contrast, the theoretical analysis on SpFDEs like problem
\eqref{strongp} remains relatively scarce. This is attribute to
the fact that fractional differential equations involve mathematical difficulties that are not
present in the analysis of the canonical second-order elliptic equations. In particular,
the fractional differential operators in the model \eqref{strongp} are nonlocal. 
In a series of works \cite{ErvinHeuerRoop:2007, ErvinRoop:2006,ErvinRoop:2007},
Ervin and Roop presented a first rigorous analysis of the well-posedness of the weak formulation
of the stationary SpFDE with Riemann-Liouville fractional derivatives via their
relation to fractional-order Sobolev spaces. The paper \cite{ErvinRoop:2006} also provided
an optimal error estimate for the Galerkin finite element method under the assumption
that the solution has full regularity, namely, $\|u \|_{H^\al\II} \le c \|f \|_{L^2\II}$.
Unfortunately, such regularity is not justified in general (see Theorem \ref{thm:regrl} and Remark \ref{rmk:regrl}).
Recently, Wang and Yang \cite{WangYang:2013} generalized the analysis to the case of
fractional-order derivatives involving a variable coefficient, analyzed the regularity of the solution
in H\"{o}lder spaces, and established the well-posedness of a Petrov-Galerkin formulation.
However, the discrete inf-sup condition was not established and hence an error estimate of
the discrete approximations was not provided.

The goal of this work is to: (1) revisit the variational formulation for problem \eqref{strongp} for
both cases of Riemann-Liouville and Caputo fractional derivatives, (2) establish the variational
stability and shift theorems, and (3) develop relevant finite element
analysis and derive error estimates of the discrete approximations expressed in terms of the
smoothness of the right-hand side only. The Caputo case, which is very natural with its convenient treatment
of boundary conditions, has not been considered earlier. As we shall see in Remark \ref{rmk:extcaputo} later,
there is an obvious difficulty in the attempt to extend the Caputo fractional derivative operator
defined below in \eqref{Caputo} from functions in $C^n\II$ to fractional order Sobolev spaces $H^s\II$.
In this paper we have developed a strategy that allows us to overcome this difficulty.

The study of problem \eqref{strongp} from a variational point of view can be summarized as follows.
The Riemann-Liouville case leads to a nonsymmetric but coercive bilinear form on the space $H_0^{\alpha/2}
\II$ (see Theorem \ref{t:sourceRL}). This fact for the case $q=0$ has been established earlier by
Ervin and Roop in \cite{ErvinRoop:2006}. The Caputo case requires a test space $V$ that is different
from the solution space $U$ and involves a nonlocal integral constraint; see \eqref{V-space}. Further,
we establish the well-posedness of the variational formulations, and more importantly, we establish
regularity pickup of the weak solutions of \eqref{strongp}, cf. Theorems \ref{thm:regrl} and
\ref{thm:regcap}, as well as that of the weak solutions to the adjoint problem. We note that the
solution regularity has only been assumed earlier, e.g., in the error analysis of Ervin and Roop
\cite{ErvinRoop:2006}. The regularity pickup is essential for proving (optimal) convergence rates for
the finite element method.

The rest of the paper is organized as follows. In Section \ref{sec:prelim} we recall preliminaries
on fractional calculus, and study fractional derivatives as operators on fractional Sobolev spaces.
Then in Section \ref{sec:strongsol}, we explicitly construct the strong solution representation via
fractional-order integral operators. The continuous variational formulations are developed in Section
\ref{sec:stability}, and their well-posedness is also established. Then in Section \ref{sec:fem}, the
stability of the discrete variational formulations is studied, and error estimates of the finite
element approximations are provided. Finally, in Section \ref{sec:num} some illustrative numerical
results are presented to confirm the error estimates. Throughout, the notation $c$, with or
without subscript, refers to a generic positive constant which can take different values at
different occurrences, but it is always independent of the solution $u$ and the mesh size $h$, and
$n$ denotes a fixed positive integer.

\section{Fractional differential operators on fractional Sobolev spaces}
\label{sec:prelim}

We first briefly recall two common fractional derivatives, i.e., Caputo and  Riemann-Liouville
fractional derivatives. For any positive non-integer  real number $\beta$ with $n-1 < \beta < n$,
the (formal) left-sided Caputo fractional derivative of order $\beta$ is defined by (see, e.g.,
\cite[pp. 92]{KilbasSrivastavaTrujillo:2006},
\cite{Podlubny_book})
\begin{equation}\label{Caputo}
  \DDC 0 \be u = {_0\hspace{-0.3mm}I^{n-\beta}_x}\bigg(\frac {d^nu} {d
    x^n }\bigg).
\end{equation}
and the (formal) left-sided Riemann-Liouville fractional derivative of order $\be$
is defined by \cite[pp. 70]{KilbasSrivastavaTrujillo:2006}:
\begin{equation}\label{Riemann}
  \DDR0\be u =\frac {d^n} {d x^n} \bigg({_0\hspace{-0.3mm}I^{n-\beta}_x} u\bigg) .
\end{equation}
In the definitions \eqref{Caputo} and \eqref{Riemann}, $_0\hspace{-0.3mm}I^{\gamma}_x$ for $\gamma>0$
is the left-sided Riemann-Liouville fractional integral operator of order $\gamma$ defined by
\begin{equation*}
 ({\,_0\hspace{-0.3mm}I^\gamma_x} f) (x)= \frac 1{\Gamma(\gamma)} \int_0^x (x-t)^{\gamma-1} f(t)dt,
\end{equation*}
where $\Gamma(\cdot)$ is Euler's Gamma function defined by $\Gamma(x)=\int_0^\infty t^{x-1}e^{-t}dt$.
As the order $\gamma$ approaches 0, we can identify the operator $_0\hspace{-0.3mm}I_x^\gamma$ with the identity operator
\cite[pp. 65, eq. (2.89)]{Podlubny_book}. We note that the fractional integral operator ${_0\hspace{-0.3mm}I^\gamma_x}$
satisfies a semigroup property, i.e., for $\gamma,\delta> 0$ and smooth $u$, there holds \cite[Lemma 2.3, pp. 73]{KilbasSrivastavaTrujillo:2006}
\begin{equation} \label{semig}
{_0\hspace{-0.3mm}I_x^{\gamma+\delta}}u={_0\hspace{-0.3mm}I_x^\gamma} {_0\hspace{-0.3mm}I_x^\delta} u.
\end{equation}
This identity extends to the space $L^2\II$ by a density argument.

Clearly, the fractional derivatives $\DDC0\beta$ and $\DDR0\beta$  are well defined for functions in
$C^n[0,1]$ and are related to each other by the formula (cf. \cite[pp. 91, eq. (2.4.6)]{KilbasSrivastavaTrujillo:2006})
\begin{equation} \label{rl-c}
 \DDC0\beta u =  {\DDR0\beta u}-\sum_{k=0}^{n-1} \frac
{u^{(k)}(0)} {\Gamma(k-\beta  +1)} x^{k-\beta}.
\end{equation}
This relation will be used to construct the solution representation in the Caputo case in Section \ref{sec:strongsol}.

The right-sided versions of fractional-order integrals and derivatives are defined analogously, i.e.,
\begin{equation*}
  ({_x\hspace{-0.3mm}I^\gamma_1} f) (x)= \frac 1{\Gamma(\gamma)}\int_x^1 (x-t)^{\gamma-1}f(t)\,dt,
\end{equation*}
and
\begin{equation*}
\begin{aligned}
  \DDC1\beta u = (-1)^n {_x\hspace{-0.3mm}I^{n-\beta}_1}\bigg(\frac {d^nu} {d x^n }\bigg),
\quad \quad  \DDR1\beta u =(-1)^n\frac {d^n} {d x^n} \bigg({_x\hspace{-0.3mm}I^{n-\beta}_1} u\bigg) .
\end{aligned}
\end{equation*}
The formula analogous to \eqref{rl-c} is \cite[pp. 91, eq. (2.4.7)]{KilbasSrivastavaTrujillo:2006}
\begin{equation} \label{rl-c1}
{ \DDC1\be u }={\DDR1\beta u}-\sum_{k=0}^{n-1} \frac
{(-1)^k u^{(k)}(1)} {\Gamma(k-\be+1)} (1-x)^{k-\beta}.
\end{equation}
Finally, we observe that for $\phi,\theta\in L^2\II$ and $\beta>-1$, there holds
\begin{equation*}
   \begin{aligned}
     \int_0^1 \int_0^x (x-t)^\beta |\phi(t) \theta(x)|\, dt dx&\le \frac 12\int_0^1 \int_0^x (x-t)^{\beta} \big[ \phi(t)^2+\theta(x)^2\big] \, dt dx\\
     &\le  \frac 1 {2 (1+\beta)} \big[\|\phi\|_{L^2\II}^2 + \|\theta\|_{L^2\II}^2\big].
   \end{aligned}
\end{equation*}
Here the second line follows from the inequality $\int_0^x(x-t)^\beta dt\le\int_0^1(1-t)^\beta dt
= \frac{1}{1+\beta}$. Therefore, Fubini's Theorem implies the following useful change of integration
order formula (cf. also \cite[Lemma 2.7]{KilbasSrivastavaTrujillo:2006}):
\begin{equation}\label{eqn:intpart}
({_0\hspace{-0.3mm}I_x^\beta} \phi,\theta) = (\phi,{_x\hspace{-0.3mm}I_1^\beta} \theta),
\qquad \hbox{ for all } \phi,\theta\in L^2\II.
\end{equation}

In order to investigate the model \eqref{strongp}, we study these operators in fractional
Sobolev spaces. This can be seen from the fact that the finite element method for the
second-order elliptic problems is most conveniently analyzed in Sobolev spaces. In particular, optimal
convergence rates with respect to the data regularity hinge essentially on the
Sobolev regularity of the solutions. Analogously, the regularity properties of solutions to
\eqref{strongp} are most naturally studied on fractional order Sobolev spaces.
However, this requires the extension of the formal definition of the fractional derivatives.

To this end, we first introduce some function spaces. For any $\beta\ge 0$, we denote
$H^\beta\II$ to be the Sobolev space of order $\beta$ on the unit interval $\II$
(see, e.g., \cite{grisvard}), and $\Hd \beta\II$ to be the set of
functions in $H^\beta\II$ whose extension by zero to $\RR$ are in $H^\beta(\RR)$.
These spaces are characterized in \cite{grisvard}.
For example, it is known that for $\beta\in(0,1)$, the space $\Hd \beta\II$ coincides with
the interpolation space $[L^2\II,H_0^1\II]_\beta$. It is important for our further
study to note that $ \phi \in \Hd \beta\II$, $ \beta >\frac32$, satisfies the boundary conditions
$\phi(0)=\phi^\prime(0)=0$ and $\phi(1)=\phi^\prime(1)=0$.

Analogously, we define $\Hdi 0 \beta$
(respectively, $\Hdi 1 \beta$) to be the set of functions $u$ whose extension by zero
$\tu$ is in $H^\beta(-\infty,1)$ (respectively, $H^\beta(0,\infty)$). Here for $u\in \Hdi 0
\beta$, we set $\|u\|_{\Hdi 0\beta}:=\|\tu\|_{H^\beta(-\infty,1)}$ with the analogous definition
for the norm in $\Hdi 1 \beta$.
Let $\Cd0^\beta\II$ (respectively, $\Cd1^\beta\II$)
denote the set of functions in
$v\in C^\infty[0,1]$ satisfying $v(0)=v^\prime(0)=\ldots=v^{(k)}(0)=0$
(respectively, $v(1)=v^\prime(1)=\ldots=v^{(k)}(1)=0$) for
$k<\beta-1/2$. It is not hard to see that the spaces
$\Cd0^\beta\II$ and $\Cd1^\beta\II$ are dense in
the spaces $\Hdi 0 \beta$ and  $\Hdi1 \beta$, respectively.
Throughout, we denote by $\tu$ the extension of $u$ by zero to
$\RR$.

The first result forms the basis for our investigation of fractional derivatives in fractional
Sobolev spaces. It was known in \cite[Lemma 2.6]{ErvinRoop:2006}, and we provide an
alternative proof since the idea will be used later.

\begin{theorem} \label{l:extendR}
For any $\beta\in (n-1,n)$, the operators $\DDR 0 \beta u$ and $\DDR 1 \beta u$ defined for
$u\in C^\infty_0\II$ extend continuously to operators (still denoted by $\DDR 0 \beta u$
and $\DDR 1 \beta u$) from $\Hd \beta \II$ to $L^2\II$.
\end{theorem}
\begin{proof}
We first consider the left sided case. For $v\in C_0^\infty(\RR)$, we define
\begin{equation}
\big (\DDRI0\beta v\big)(x)
    = \frac{d^n}{dx^n}\bigg(\int_{-\infty}^x (x-t)^{\beta-n} v(t)\, dt\bigg). 
\label{tdbeta}
\end{equation}
We note the following identity for the Fourier transform
\begin{equation}\label{fident}
\cF\big ( \DDRI 0 \beta v\big ) (\omega) =(-i\omega)^\beta \cF(v)(\omega), \quad \mbox{i.e.},
\quad \DDRI 0 \beta v (x ) = \cF^{-1} \big ((-i\omega)^\beta \cF(v)(\omega) \big),
\end{equation}
which holds for $v\in C_0^\infty(\RR)$ (cf. \cite[pp. 112, eq. (2.272)]{Podlubny_book} or
\cite[pp. 90, eq. (2.3.27)]{KilbasSrivastavaTrujillo:2006}). It follows from Planchel's theorem that
$$ \|  \DDRI 0\beta v\|_{L^2(\RR)} =
\| \cF\big ( \DDRI 0 \beta v\big)\|_{L^2(\RR)}\le c\|v\|_{H^\beta(\RR)}.$$
Thus, we can contiuously extend $\DDRI 0\beta  $ to an operator from
$H^\beta(\RR)$ into $L^2(\RR)$ by formula \eqref{fident}.

We note that for $u\in C_0^\infty\II$, there holds
\begin{equation}\label{ddext}
\DDR0\beta u = \DDRI0\beta \tu |_\II.
\end{equation}
By definition,  $u\in \Hd \beta \II$ implies that
$\tu$ is in $H^\beta(\RR)$ and hence
\begin{equation*}
\|\DDRI0 \beta \tu \|_{L^2(\RR)} \le c \|u\|_{\Hd \beta\II}.
\end{equation*}
Thus, formula \eqref{ddext} provides an extension of the operator $\DDR0\beta$
defined on $C_0^\infty\II$ to a bounded operator from
the space $\Hd \beta\II$ into $L^2\II$.

The right sided case is essentially identical except for replacing
 \eqref{tdbeta} and  \eqref{fident} with
\begin{equation*} 
\big(\DDRI1\beta v\big)(x)
        =(-1)^n \frac{d^n}{dx^n}\bigg (\int^{\infty}_x (t-x)^{\beta-n} v(t)\,dt\bigg)
\end{equation*}
and
\begin{equation*} 
\cF \big(\DDRI1\beta v\big) (\omega) =(i\omega)^\beta \cF(v)(\omega).
\end{equation*}
This completes the proof of the theorem.
\end{proof}

\begin{remark} \label{indep}
We clearly have that $\DDRI0\beta v(x) =0$ for $x<1$ when $v\in\Cd0^\beta(0,2)$ is
supported on the interval $[1,2]$. By density, this also holds for
$v\in \widetilde{H}_L^\beta (0,2)$ supported on $[1,2]$.
\end{remark}

The next result slightly relaxes the conditions in Theorem \ref{l:extendR}.
\begin{theorem} \label{l:extendR1}
For  $\beta\in (n-1,n)$, the operator $\DDR0\beta u$ defined for $u\in \Cd 0^\beta\II$
extends continuously to an operator from $\Hdi 0 \beta$ to $L^2\II$. Similarly, the
operator $\DDR1\beta $ defined for $u\in \Cd 1^\beta\II$ extends continuously to an
operator from $\Hdi 1 \beta$ to $L^2\II$.
\end{theorem}
\begin{proof}
We only prove the result for the left sided derivative since the proof for the right sided case is
identical. For a given $u\in \Hdi 0 \beta$, we let $u_0$ denote
a bounded extension of $u$ to $\Hd  \beta  (0,2)$, and then set
$$\DDR0\beta u= {\DDR0\beta} \tu_0|_\II.$$
It is a direct consequence of Remark~\ref{indep} that
 $\DDR0\beta u$ is independent of the
extension $u_0$ and coincides with the formal definition of
$\DDR0 \beta u$ when $u\in \Cd0^\beta\II$.  We obviously have
\begin{equation*}
\|\DDR0\beta u\|_{L^2\II} \le  \|\DDRI0\beta \tu_0\|_{L^2(\RR)} \le c \|u\|_{\Hdi 0 \beta }.
\end{equation*}
This completes the proof.
\end{proof}

\begin{remark}
Even though $\DDR0\beta$ makes sense as an operator from $\Hdi 0 \beta$ into
$L^2\II$, it apparently cannot be extended to useful larger spaces, for example,
$\DDR0\beta(1) = c_\beta x^{-\beta}$ is not in the space $L^2\II$, if $\beta>\frac{1}{2}$.
Accordingly, it is not obvious in what space one should seek the solutions of
\eqref{strongp}. However, it is clear that, at least for $\alpha\in (3/2,2)$,
the solution is not generally in $\Hd \alpha \II$ since such functions
satisfy the additional  boundary conditions $u^\prime(0)=u^\prime(1)=0$.
\end{remark}

\begin{remark}\label{rmk:extcaputo}
Lemma~\ref{l:caprl} below shows that for $\beta\in (0,1)$, the operators $\DDC0\beta$ and $\DDR0\beta$
coincide on the space $\Hdi 0 1$. Thus, the continuous extension of $\DDC0\beta$ defined on $\Cd
0^\beta \II$  to $\Hdi 0 \beta$ is $\DDC0\beta$. This is somewhat misleading since the
extension does not necessarily coincide with the formal definition, e.g., $1$ is in
$\Hdi 0 \beta$ for $\beta\in (0,1/2)$,  $ \DDR0\beta (1)=c_\beta x^{-\beta}$  while
$\DDC0\beta (1) =  0$. The problem is that the continuous extension of the Caputo
derivative $\DDR0\beta$ from $\Cd0^\beta\II$ to $\Hdi 0 \beta$ does not necessarily
coincide with the formal definition of $\DDC0\beta$ even on functions for which the formal
definition does make sense.
\end{remark}

\section{Strong solutions of fractional order equations}\label{sec:strongsol}

The problem we now face is to make some sense out of the model \eqref{strongp}. We start by studying
the case of $q=0$. We shall construct functions whose fractional derivatives make sense and satisfy
\eqref{strongp}. The following smoothing property of fractional integral operators $_0I_x^\beta $
and $_xI_1^\beta$ will play a central role in the construction. We note that the smoothing property
of $_0I_x^\beta$ in $L^p\II$ spaces and H\"{o}lder spaces was
studied extensively in \cite[Chapter 1, \S 3]{SamkoKilbasMarichev:1993}.

\begin{theorem}  \label{l:i0reg}
For any $s,\beta\geq 0$, the operators $_0I_x^\beta $ and $_xI_1^\beta$ are bounded maps from $\Hd s \II$ into $\Hdi 0
{s+\beta}$ and $\Hdi1{s+\beta}$, respectively.
\end{theorem}
\begin{proof}
The key idea of the proof is to extend $f\in \Hd s \II$ to a function $\tf \in  \Hd s (0,2)$ whose moments
up to $(k-1)$th order vanish with $k>\beta-1/2$. To this end, we employ
orthogonal polynomials $\{p_0,p_1,\ldots,p_{k-1}\}$ with respect to the inner
product $\langle\cdot,\cdot\rangle$ defined by
\begin{equation*}
\langle u,v\rangle=\int_1^2 ((x-1)(2-x))^l u(x)v(x) \, dx,
\end{equation*}
where the integer $l$ satisfies $l>s-1/2$ so that
\begin{equation*}
((x-1)(2-x))^l p_i \in \Hd s (1,2),\qquad i=0,\ldots,k-1.
\end{equation*}
Then we set $w_j=\gamma_j ((x-1)(2-x))^l p_j$ with $\gamma_j$ chosen so that
\begin{equation*}
   \int_1^2 w_j p_j\, dx = 1 \quad \text{so that} \quad \int_1^2 w_j p_l\, dx = \delta_{j,l},
   \quad j,l=0, \dots, k-1.
\end{equation*}
Next we extend both $f$ and $w_j$, $j=0,\ldots, k-1$ by zero to $(0,2)$ by setting
\begin{equation*}
  f_e = f-\sum_{j=0}^{k-1} \bigg(\int_0^1 f p_j\, dx\bigg) w_j.
\end{equation*}
The resulting function $f_e$ has vanishing moments for $j=0,\ldots,k-1$ and by construction
it is in the space $\Hd s (0,2)$. Further, obviously there holds the inequality
$\|f_e\|_{L^2(0,2)}\leq C\|f\|_{L^2\II}$, i.e., the extension is bounded in $L^2\II$. As usual,
we denote by $\tf_e$ the extension of $f_e$ to $\RR$ by zero.

Now on the interval $(0,1)$, there holds ${_0I_x^\beta} f = {_{-\infty}I_x^\beta} \tf_e$, where
\begin{equation*}
_{-\infty}I_x^\beta(\tf)(x)=  \frac 1 {\Gamma(\beta)} \int_{-\infty}^x
(x-t)^{\beta-1} \tf(t)dt.
\end{equation*}
Meanwhile, we have (see \cite[pp. 112, eq. (2.272)]{Podlubny_book} or
\cite[pp. 90, eq. (2.3.27)]{KilbasSrivastavaTrujillo:2006})
\begin{equation} \label{fident1}
\cF ({_{-\infty}I_x^\beta} \tf_e) (\omega) =(-i\omega)^{-\beta} \cF(\tf_e)(\omega)
\end{equation}
and hence by Planchel's theorem
\begin{equation}\label{oint}
  \|{_{-\infty}I^\beta_x}\tf_e\|^2_{L^2(\RR)} =
  \int_\RR |\omega|^{-2\beta} |\cF(\tf_e)(\omega)|^2 \, d\omega.
\end{equation}
We note that by Taylor expansion at $0$, there holds
\begin{equation*}
  \begin{aligned}
    e^{-i\omega x} - & 1-(-i\omega) x -\cdots- \frac {(-i\omega)^{k-1}x^{k-1}}{(k-1)!} \\
   =& \frac{(-i\omega x)^k}{k!}+\frac{(-i\omega x)^{k+1}}{(k+1)!}+\frac{(-i\omega x)^{k+2}}{(k+2)!}+\cdots = {(-i\omega)^{k}} I_0^k (e^{-i\omega x}).
  \end{aligned}
\end{equation*}
Clearly, there holds $|I_0^k (e^{-i\omega x})|\le I_0^k (1) = \frac {x^k}{k!}$.
Since the first $k$ moments of $\tilde f_e$ vanish, multiplying the above identity by
$\widetilde f_e$ and integrating gives
$$\cF(\tf_e)(\omega) = {(-i\omega)^{k}} \frac 1 {\sqrt{2\pi}} \int_\RR I_0^k
(e^{-i\omega x}) \tf_e(x) \, dx,$$
and upon noting $\mathrm{supp}(\widetilde{f}_e)\subset(0,2)$, consequently,
$$|\cF(\tf_e)(\omega)|\le \frac {2^{k}} {\sqrt \pi k!}  |\omega|^k
\|f_e \|_{L^2(0,2)} \le c |\omega|^k \| f \|_{L^2\II}.$$
We then have
\begin{equation*}
  \begin{aligned}
    \|{_{-\infty}I_x^\beta}\tf_e\|^2_{H^{\beta+s}(\RR)} &=
      \int_\RR (1+|\omega|^2)^{\beta+s} |\omega|^{-2\beta}|\cF(\tf_e)(\omega)|^2 d\omega\\
    &\le c_1 \|f\|_{L^2\II}\int_{|\omega|<1}   |\omega|^{-2\beta+2k}d \omega
    +c_2\int_{|\omega|>1}  |\omega|^{2s}|\cF(\tf_e)(\omega)|^2d\omega \\
    &\le c\|f\|_{\Hd s  \II}.
  \end{aligned}
\end{equation*}
The desired assertion follows from this and the trivial inequality $\|{_0I_x^\beta} f\|_{H^{\beta+s}\II}
\leq\|{_{-\infty}I^\beta_x}\tf_e\|^2_{H^{\beta+s}(\RR)}$.
\end{proof}

\begin{remark}\label{rmk:i0}
By means of the extension in Theorem \ref{l:extendR1}, the operator
$_0I_x^\beta$ is bounded from $\Hdi0 s$ to $\Hdi0{\beta+s}$,
and $_xI_1^\beta$ is bounded from $\Hdi1 s$ to $\Hdi1{\beta+s}$.
\end{remark}

A direct consequence of Theorem \ref{l:i0reg} is the following useful corollary.
\begin{corollary}\label{phi-0}
Let $\gamma$ be non-negative. Then the functions $x^\gamma$ and $(1-x)^{\gamma}$ belong
to $\Hdi 0 {\beta}$ and $\Hdi 1 {\beta}$, respectively, for any $0\le \beta <\gamma+1/2$.
\end{corollary}
\begin{proof} We note the relations
$x^\gamma = c_\gamma \, {_0I_x^\gamma}(1)$ and
 $(1-x)^\gamma=c_\gamma \, {_xI_1^\gamma} (1)$.
The desired result follows from Theorem~\ref{l:i0reg} and the fact that $1\in \Hdi0 \delta$ and $1\in \Hdi 1
\delta$ for any $\delta\in [0,1/2)$.
\end{proof}

Now we are in a position to construct the strong solutions to problem \eqref{strongp}
in the case of a vanishing potential, i.e., $q=0$. We first consider the Riemann-Liouville
case. Here we fix $f\in L^2\II$ and set $g={_0I_x^\alpha} f \in \Hdi 0 \alpha$.  By
Theorem~\ref{l:extendR1}, the fractional derivative $\DDR0\alpha g$ is well defined.
Now in view of the semigroup property \eqref{semig}, we deduce
\begin{equation*}
{_0I_x^{2-\alpha}} g = {_0I_x^{2-\alpha}}{_0I_x^\alpha} f = {_0I_x^2}f\in \Hdi 0 2.
\end{equation*}
It is straightforward to check that $({_0I_x^2}f)^{\prime\prime}=f$ holds for smooth
$f$ and hence also on $L^2\II$ by a density argument. This implies that $\DDR0\alpha g=f$. (This
relation is reminiscent of the fundamental theorem in calculus.) We thus find that
\begin{equation} \label{strongrl}
u= -{_0I_x^\alpha} f +({_0I_x^\alpha} f)(1) x^{\alpha -1}
\end{equation}
is a solution of \eqref{strongp} in the Riemann-Louiville case when $q=0$ since
$u$ satisfies the correct boundary conditions and $\DDR0\alpha x^{\alpha-1}=
(c_\alpha x)^{\prime\prime}=0$.

Next we consider the Caputo case. To this end, we choose $\beta\geq0$ so that $\alpha+
\beta\in (3/2,2)$. For smooth $u$ and $\alpha\in (1,2)$, the Caputo and Riemann-Liouville
derivatives are related by (cf. \eqref{rl-c})
\begin{equation*}
\DDC0\alpha u={\DDR0\alpha} u- \frac{u(0)} {\Gamma(1-\alpha)} x^{-\alpha}-\frac
{u^{\prime}(0)} {\Gamma(2-\alpha)} x^{1-\alpha}.
\end{equation*}
Applying this formula to $g={_0I_x^\alpha} f$ for $f\in \Hd \beta \II$ (in this case $g$
is in $\Hdi 0{\alpha+\beta}$ by Theorem \ref{l:i0reg} and hence satisfies $g(0)=g^\prime(0)=0$)
shows that $ \DDC0\alpha g$ makes sense and equals $ \DDR0\alpha g=f$. Thus a solution $u$
of \eqref{strongp} in the Caputo case with $q=0$ is given by
\begin{equation} \label{strongc}
u= -{_0I_x^\alpha} f +   ({_0I_x^\alpha} f)(1) \, x.
\end{equation}

\begin{remark}
The solution representation of the model \eqref{strongp} with a Caputo fractional derivative $\DDC0\alpha u$
and low-regularity source term $f\in \Hd \beta\II $ such that $\alpha+\beta\leq3/2$ remains unclear.
\end{remark}

\section{Variational formulations of the fractional derivative problems}\label{sec:stability}

In this part we develop variational formulations of problem \eqref{strongp}
and establish shift theorems for the variational solution. The shift
theorems show the regularity pickup for the variational solution,
which will be essential for deriving error estimates in Section \ref{sec:fem}. We
shall first consider the case of $q=0$ where we have an explicit
representation of the solutions.

\subsection{Derivation of the variational formulations}
The starting point of the derivation of the variational formulations is the following lemma.

\begin{lemma} \label{l:caprl}
For $u\in \Hdi 0 1$  and $\beta\in (0,1)$,
\begin{equation*}
\DDR 0\beta u = {_0I_x^{1-\beta}}(u^\prime).
\end{equation*}
Similarly, for $u\in \Hdi 1 1$  and $\beta\in (0,1)$,
\begin{equation*}
   \DDR1\beta u =-{_x I_1^{1-\beta}}(u^\prime).
\end{equation*}
\end{lemma}
\begin{proof}
It suffices to prove the result for the left sided derivative $\DDR0\beta$ since
the right sided case is analogous. For $u\in \Cd 0^\beta  \II$, \eqref{rl-c} implies that
\begin{equation}\label{smid}
\DDR0\beta u = \DDC0\beta u = {_0I_x^{1-\beta}}(u^\prime).
\end{equation}
Theorem~\ref{l:extendR1} implies that the left hand side extends to a continuous
operator on $\Hdi 0 \beta$ and hence $\Hdi 0 1$ into $L^2\II$. Meanwhile,
Theorem~\ref{l:i0reg} implies that the right hand side of \eqref{smid} extends to
a bounded operator from $H^ 1\II$ into $L^2\II$.  The lemma now follows by density.
\end{proof}

We next motivate the variational formulation in the Riemann-Liouville
case. Upon taking $u$ as in \eqref{strongrl}, $g={_0I_x^\alpha}f\in \Hdi0\alpha$
and $v\in C_0^\infty\II$, Lemma~\ref{l:caprl} implies
\begin{equation}\label{rlvar}
\begin{aligned}
  (\DDR0\alpha u,v) &= -\big(({_0I^{2-\alpha}_x} g)^{\prime\prime},v\big) =\big(({_0I^{2-\alpha}_x} g)^{\prime},v^\prime\big) \\
   &=\big({\DDR0{\alpha-1}} g,v^\prime\big) = ({_0I_x^{2-\alpha}}g',v'),
\end{aligned}
\end{equation}
where we have used the identity $\DDR0{\alpha} x^{\alpha-1}=0$
in the second step. 
Now the semigroup property \eqref{semig} and the change of integration order formula \eqref{eqn:intpart} yield
\begin{equation}
(\DDR0\alpha u,v)= \big({_0I^{1-\alpha/2}_x} g',{_xI_1^{1-{\alpha/2}}}v^\prime\big).
\label{rl-first}
\end{equation}
Since $g\in \Hdi0{\alpha}$ and $v\in\Hd 1\II$, we can apply Lemma~\ref{l:caprl} again to conclude
\begin{equation*}
  (\DDR0\alpha u,v) = -(\DDR0{\alpha/2} g,\ {\DDR1{\alpha/2}}v).
\end{equation*}
Further direct computation shows that $(\DDR0{\alpha/2} x^{\alpha-1},{\DDR1{\alpha/2}}v) = (\DDR0{\alpha-1}x^{\alpha-1},v')=0$.
Consequently,
\begin{equation}\label{adef}
   A(u,v):=-(\DDR0\alpha u,v)=-(\DDR0{{\alpha/2}}u,\,\DDR1{{\alpha/2}}v).
\end{equation}
Later in Lemma \ref{coercive}, we shall show that $A(u,v)$ is a coercive and bounded bilinear form on the space
$\Hd {{\alpha/2}} \II$, and thus our variational formulation of problem \eqref{strongp} in the
Riemann-Liouville case (with $q=0$) is to find $u\in U:=\Hd  {{\alpha/2}} \II$ satisfying
\begin{equation*}
A(u,v) = (f,v), \qquad \hbox{ for all }v\in U.
\end{equation*}

We next consider the Caputo case. Following the definition of the solution in \eqref{strongc}
in Section \ref{sec:strongsol}, we choose $\beta\geq0$ so that $\alpha+\beta\in (3/2,2)$ and $f\in \Hd \beta
\II$. By Theorem~\ref{l:i0reg}, the function $g={_0I_x^\alpha} f$ lies in $\Hdi 0 {\alpha+\beta}$ and hence
$g^\prime(0)=0$. Differentiating both sides of \eqref{strongc} and setting $x=0$ yields the identity
$u'(0)=({_0I_x^\alpha f})(1)$, and thus, the solution representation \eqref{strongc} can be rewritten as
\begin{equation*}
u= -{_0I_x^\alpha} f + x  u^\prime(0).
\end{equation*}
Meanwhile, for $v\in \Cd1^{{\alpha/2}} \II$, \eqref{rlvar} implies
\begin{equation*}
-(\DDR0\alpha g,v) = \big(({_0I^{2-\alpha}_x} u)^{\prime\prime},v\big)
=\big(({_0I^{2-\alpha}_x} g)^{\prime},v^\prime\big),
\end{equation*}
where we have again used the relation $g^\prime(0)=0$ in the last step. Analogous to the derivation
of \eqref{rl-first} and the identities $\DDC0\alpha g={\DDR0\alpha g}$ and $\DDC0\alpha x=0 $, we then have
\begin{equation}\label{cv1}
-(\DDC0\alpha u ,v) = (\DDR0\alpha g,v)
=-A(g,v) = A(u,v) - u^\prime(0) A(x,v).
\end{equation}
The term involving $u^{\prime}(0)$ cannot appear in the variational formulation in $\Hd {\alpha/2}\II$.
To circumvent this issue, we reverse the preceding steps and arrive at
\begin{equation*}
A(x,v) =  \big( {_0I^{2-\alpha}_x} 1,v^\prime\big)
= (\Gamma(3-\alpha))^{-1}  (x^{2-\alpha},v^\prime )= -(\Gamma(2-\alpha))^{-1}  (x^{1-\alpha},v ).
\end{equation*}
Hence, in order to get rid of the troublesome term $- u^\prime(0) A(x,v)$ in \eqref{cv1}, we
require our test functions to satisfy the integral condition $(x^{1-\alpha},v)=0$. Thus the
variational formulation of \eqref{strongp} in the Caputo case (with $q=0$) is to find  $u\in U$ satisfying
$$A(u,v) = (f,v), \qquad \hbox{ for all } v\in V,$$
with the test space
\begin{equation}\label{V-space}
   V=\left\{ \phi\in \Hdi 1{{\alpha/2}} \ :\ (x^{1-\alpha},\phi)=0\right\}.
\end{equation}
We shall show that this is a well-posed problem  with a unique solution
(for any $f\in L^2\II$).

In the rest of this section, we discuss the stability of the variational formulations.
Throughout, we denote by $U^*$ (respectively $V^*$) the set of bounded linear functionals
on $U$ (respectively $V$), and slightly abuse $\langle\cdot,\cdot\rangle$ for duality pairing
between $U$ and $U^*$ (or between $V$ and $V^*$). Further, we will denote by $\|\cdot\|_U$
the norm on the space $U$ etc.

\begin{remark}  We have seen that when $f$ is in $\Hdi 0 \beta$, with $\al + \be \in({3/2},2)$,
then the solution $u$ constructed by \eqref{strongc} satisfies the variational equation and
hence coincides with the unique solution to the variational equation. This may not be the case
when $f$ is only in $L^2\II$.  Indeed, for $\alpha$ in $(1,{3/2})$, the function $f=x^{1-\alpha}$
is in $L^2\II$. However, the variational solution (with $q=0$) in this case is $u=0$ and clearly does not
satisfy the strong-form differential equation $\DDC0\alpha u = f$.
\end{remark}

\subsection{Variational stability in Riemann-Liouville case}
We now establish the stability of the variational formulations.
The following lemma implies variational stability in the Riemann-Liouville case with $q=0$.
The result is well known \cite{ErvinRoop:2006}, and the proof is provided only for completeness.

\begin{lemma} \label{coercive}  Let $\alpha$ be in $(1,2)$.
Then there is a positive constant $c=c(\alpha)$ satisfying
\begin{equation}
c\| u\|_{\Hd {{\alpha/2}}\II} ^2 \le -(\DDR0{{\alpha/2}}u, \,
\DDR1{{\alpha/2}}u)= A(u,u), \qquad \hbox{ for all } u\in
\Hd  {{\alpha/2}} \II.
\label{coer1}
\end{equation}
\end{lemma}
\begin{proof}
Like in the proof of Theorem~\ref{l:extendR}, we find for
$u\in C_0^\infty \II$,
\begin{equation}
\begin{aligned}
  -(\DDR0{{\alpha/2}}u,\,\DDR1{{\alpha/2}}u)&= -\int_{-\infty}^\infty (i\omega)^{\alpha}|\cF(\tilde u)|^2\, d\omega
  =2 \cos((1-\tfrac{\alpha}{2})\pi)  \int_0^\infty \omega^{\alpha}|\cF(\tilde u)|^2\, d\omega\\
  &= \cos((1-\tfrac{\alpha}{2})\pi) \int_{-\infty}^\infty |\omega|^{\alpha}
|\cF(\tilde u)|^2\, d\omega.
\end{aligned}
\label{semb}
\end{equation}
Now, suppose that there does not exist a constant satisfying
\eqref{coer1}. Then, by the compact embedding of $\Hd {{\alpha/2}} \II$
into $L^2\II$, there is a sequence  $\{u_j\}
\subset \Hd {{\alpha/2}} \II$ with $\|u_j\|_{ \Hd {{\alpha/2}} \II} =1$, $u_j$
convergent to some $u\in L^2\II$ and
satisfying
\begin{equation}\label{ujb}
\| u_j\|_{\Hd {{\alpha/2}}\II}^2 > -j  (\DDR0{{\alpha/2}}
u_j,\, \DDR1{{\alpha/2}}u_j).
\end{equation}
It follows  from \eqref{semb}, \eqref{ujb} and the convergence of the sequence $\{u_j\}$ in
$L^2\II$ that $\{\cF(\tu_j)\}$ is a Cauchy sequence and hence converges to $\cF(\tu)$ in the
norm $( \int_{-\infty}^\infty(1+\omega^2)^{\alpha} |\cdot|^2\,d\omega)^{1/2}$. This implies
that $u_j$ converges to $u$ in $\Hd {\alpha/2} \II$. Theorem~\ref{l:extendR} and \eqref{ujb}
implies that $-(\DDR0{{\alpha/2}} u,\, \DDR1{{\alpha/2}}u)=0$, from which it follows that
$\cF(\tu)=\tu=0$. This contradicts the assumption $\|u_j\|_{\Hd{{\alpha/2}}\II}=1$ and
completes the proof.
\end{proof}

We now return to the problem with $q\neq 0$ and define $$a(u,v) = A(u,v) +(qu,v).$$
To this end, we make the following uniqueness assumption on the bilinear form.
\begin{assumption} \label{ass:riem}
Let the bilinear form $a(u,v)$ with $u,v\in U$ satisfy
\begin{itemize}
 \item[{$\mathrm{(a)}$}]  The problem of finding $u \in U$ such that $a(u,v)=0$ for all $v \in U$
           has only the trivial solution $u\equiv 0$.
 \item[{$(\mathrm{a}^\ast)$}] The problem of finding $v \in U$ such that $a(u,v)=0$ for all $u \in U$
    has only the trivial solution $v\equiv 0$.
\end{itemize}
\end{assumption}

We then have the following existence result.

\begin{theorem}[Riemann-Liouville derivative] \label{t:sourceRL}
Let Assumption \ref{ass:riem} hold and $q\in L^\infty\II$. Then for any
given $F\in U^*$, there exists a unique solution $u\in U$ solving
\begin{equation}\label{varrl}
a(u,v)= \langle F,v \rangle, \qquad \hbox{ for all } v\in U.
\end{equation}
\end{theorem}
\begin{proof}
The proof is an application of the Petree-Tartar lemma \cite[pp. 469, Lemma A.38]{ern-guermond}.
To this end, we define respectively $S:U\rightarrow U^*$ and $T:U\rightarrow U^*$ by
\begin{equation*}
  \langle Su,v \rangle =a(u,v),\qquad \hbox{and} \qquad
  \langle Tu,v \rangle =-(qu,v),\qquad \hbox{ for all } v\in U.
\end{equation*}
Assumption \ref{ass:riem}(a) implies that $S$ is injective.
Further, Lemma~\ref{coercive} implies
$$\|u\|_U^2\le c A(u,u)= c (\langle Su ,u \rangle + \langle Tu,u\rangle )
\le c(\|Su \|_{U^*}+\| Tu \|_{U^*})\|u\|_U, \qquad \hbox{ for all } u\in U.$$
Meanwhile, the compactness of $T$
follows from the fact $q\in L^\infty\II$ and the compactness of $U$ in $L^2\II$.
Now the Petree-Tartar lemma directly implies that there exists a constant $\gamma>0$
satisfying
\begin{equation}\label{inf-suprl}
\gamma  \|u\|_U \le \sup_{v\in U } \frac {a(u,v)} {\|v\|_U}.
\end{equation}
This together with Assumption \ref{ass:riem}$(\hbox{a}^\ast)$ shows that
$S:U\rightarrow U^* $ is bijective, i.e., there is a unique solution
of \eqref{varrl} (see, e.g. \cite[Corollary A.45]{ern-guermond}).
\end{proof}

We now show that the variational solution $u$ in Theorem \ref{t:sourceRL}, in fact, is a strong
solution when $\langle F,v\rangle = (f,v) $ for some $f\in L^2\II$.  We consider the problem
\begin{equation}\label{rl-qneq0}
-\DDR0{\alpha} w = f-qu.
\end{equation}
A strong solution is given by \eqref{strongrl} with a right hand side $\widetilde{f}=f-qu$.
It satisfies the variational equation and hence coincides with the unique variational
solution. We record this result below.
\begin{theorem}\label{thm:regrl}
Let Assumption \ref{ass:riem} hold, and $q\in L^\infty\II$. Then with a right hand
side $\langle F,v\rangle=(f,v)$ for some $f\in L^2\II$, the solution $u$ to \eqref{varrl} is
in $\Hdi0{\alpha-1+\beta}\cap \Hd{{\alpha/2}}\II$ for any $\beta\in[0,1/2)$, and it satisfies
\begin{equation*}
  \|u \|_{\Hdi 0 {\al -1 +\beta}} \le c \|f\|_{L^2\II}.
\end{equation*}
\end{theorem}
\begin{proof} It follows from Theorem \ref{t:sourceRL} and Assumption \ref{ass:riem} that there
exists a solution $u\in \Hd{{\alpha/2}}\II$. Next we rewrite into \eqref{rl-qneq0} with a right
hand side $\widetilde{f}=f -qu$. In view of the fact that $q\in L^\infty\II$ and $u\in
\Hd{{\alpha/2}}\II$, there holds $qu\in L^2\II$, and hence $\widetilde{f}\in L^2\II$. Now the
desired assertion follows directly from the representation \eqref{strongrl}, Theorem \ref{l:i0reg}
and Corollary \ref{phi-0}.
\end{proof}

\begin{remark}\label{rmk:regrl}
In general, the best regularity of the solution to \eqref{strongp} with a Riemann-Liouville
fractional derivative is $\Hdi0{\alpha-1+\beta}$ for any $\beta\in[0,1/2)$,
due to the presence of the singular term $x^{\alpha-1}$. The only possibility of an improved regularity
is the case $({_0I_x^\alpha f})(1)=0$ (for $q=0$).
\end{remark}

We note that a similar result holds for the adjoint problem:
given $F\in U^*$, find $w\in U$ such that
\begin{equation} \label{adjvarrl}
a(v,w)= \langle v,F \rangle, \qquad \hbox{ for all } v\in U.
\end{equation}
Then there exists a unique solution $w\in U$ to the adjoint problem. Indeed, Assumption
\ref{ass:riem} and \eqref{inf-suprl} immediately implies that the inf-sup condition for
the adjoint problem holds with the same constant. Now, by proceeding as in \eqref{strongrl},
for $q=0$ and a right hand side $\langle F,v\rangle\equiv (f,v)$ with $f\in L^2\II$, we have
\begin{equation*}
w= -{_xI_1^\alpha} f +  ({_xI_1^\alpha} f)(0)(1-x)^{\alpha -1}.
\end{equation*}
This implies a similar regularity pickup, i.e., $w\in\Hdi 1 {\alpha-1+\beta}$, for \eqref{adjvarrl},
provided that $q\in L^\infty\II$. Further, we can repeat the arguments in the proof of Theorem
\ref{thm:regrl} for a general $q$ to deduce the regularity of the adjoint solution. We record the
observation in a remark.
\begin{remark}\label{rmk:regadjrl}
Let Assumption \ref{ass:riem} hold, and $q\in L^\infty\II$.
Then with a right hand side $\langle F,v\rangle=(f,v)$ for some $f\in L^2\II$,
the solution $w$ to \eqref{adjvarrl} is in $\Hdi1{\alpha-1+\beta}\cap \Hd{{\alpha/2}}\II$ for any
$\beta\in[0,1/2)$, and it satisfies
\begin{equation*}
  \|w\|_{\Hdi1 {\al -1 +\beta}} \le c \|f\|_{L^2\II}.
\end{equation*}
\end{remark}

\subsection{Variational stability in Caputo case}
We next consider the case of the Caputo derivative. We set $\phi_0=(1-x)^{\alpha-1}$.
By Corollary \ref{phi-0}, $\phi_0$ is in the space $\Hdi1 {\frac{\alpha}{2}}$.
Further, we observe that
\begin{equation}\label{ortho}
A(u,\phi_0)=0 \qquad \hbox{ for all } u\in U.
\end{equation}
Indeed, for $u\in \Hd 1\II$, by the change of integration order formula \eqref{eqn:intpart}, there holds
\begin{equation*}
\begin{aligned}
   A(u,(1-x)^{\alpha-1})&= -(I_0^{1-{\alpha/2}} u^\prime, \, \DDR1{{\alpha/2}}(1-x)^{\alpha-1} ) \\&= c_\alpha (
     u^\prime,I_1^{1-{\alpha/2}} (1-x)^{{\alpha/2}-1}) =c_\alpha ( u^\prime,1)=0.
 \end{aligned}
\end{equation*}

Now for a given $u\in U$, we set $v=u-\gamma_u \phi_0$, where the linear functional $\gamma_u$ is defined by
\begin{equation}\label{gammau}
\gamma_u = \frac {(x^{1-\alpha},u)} {(x^{1-\alpha},\phi_0)}.
\end{equation}
Clearly, by the norm equivalence $\|u\|_{H^{{\alpha/2}}\II} $ and
$\|u\|_{\Hdi 1 {{\alpha/2}}}$ for $\alpha\in(1,2)$, and the continuous
embedding from $H^{{\alpha/2}}\II$ into $L^\infty\II$, we have
\begin{equation*}
  |\gamma_u| \le c |(x^{1-\alpha},u)|\le c\|u\|_{L^\infty\II} \|x^{1-\alpha}\|_{L^1\II}\le c\|u\|_{\Hd {\alpha/2}\II}.
\end{equation*}
Consequently, the function $v$ satisfies $v\in H^{\alpha/2}\II$ and $v(1)=0$,
i.e., it is in the space $V$. By Lemma~\ref{coercive},
\begin{equation*}
A(u,v) = A(u,u) \ge c\|u\|_{\Hd {{\alpha/2}} \II}^2.
\end{equation*}
Finally, there also holds
\begin{equation*}
  \|v\|_{\Hdi 1 {{\alpha/2}}} \le \|u\|_{\Hd {{\alpha/2}}\II}+ c|\gamma_u|
  \le c\|u\|_{\Hd {{\alpha/2}}\II},
\end{equation*}
and thus the inf-sup condition follows immediately
\begin{equation}\label{isUV}
 \|u\|_{\Hd {{\alpha/2}}\II}\le c \sup_{v\in V} \frac {A(u,v)}{\|v\|_{\Hdi 1 {{\alpha/2}}}}, \hbox{ for all } u\in U.
\end{equation}

Now given any $v\in V$, we set $u=v-v(0)\phi_0$.  Obviously, $u$ is nonzero
whenever $v\neq 0$ and applying Lemma~\ref{coercive} we get
\begin{equation*}
   A(u,v)=A(u,u)>0.
\end{equation*}
This implies that if $A(u,v)=0$ for all $u\in U$, then $v=0$.  This and \eqref{isUV}
imply that the corresponding variational problem is stable, i.e., given $F\in V^*$,
there exists a unique $u\in U$ satisfying
$$A(u,v) = \langle F,v\rangle, \qquad \hbox{ for all } v\in V.$$

We next consider the case $q\ne0$ and like before we assume the uniqueness of the bilinear form.
\begin{assumption}\label{ass:caputo}
Let the bilinear form $a(u,v)$ with $u\in U,v\in V$ satisfy
\begin{itemize}
 \item[{$\mathrm{(b)}$}]The problem of finding $u \in U$ such that $a(u,v)=0$ for all $v \in V$
           has only the trivial solution $u\equiv 0$.
 \item[{$(\mathrm{b}^\ast)$}] The problem of finding $v \in V$ such that $a(u,v)=0$ for all $u \in U$
    has only the trivial solution $v\equiv 0$.
\end{itemize}
\end{assumption}

We then have the following existence result.

\begin{theorem}[Caputo derivative] \label{t:sourceC}
Let Assumption \ref{ass:caputo} hold and $q\in L^\infty\II$. Then for any given $F\in V^*$,
there exists a unique solution $u\in U$ to
\begin{equation}\label{varc}
a(u,v)= \langle F,v \rangle, \qquad \hbox{ for all } v\in V.
\end{equation}
\end{theorem}
\begin{proof}  The proof is similar to that  of Theorem \ref{t:sourceRL}.
In this case, we define $S:U\rightarrow V^*$ and $T:U\rightarrow V^*$ by
$$\langle Su,v \rangle =a(u,v),\qquad \hbox{and} \qquad
\langle Tu,v \rangle =-(qu,v),\qquad \hbox{ for all } v\in V.$$
Assumption \ref{ass:caputo}(b) implies that $S$ is injective.
Applying \eqref{isUV} we get for any $u\in U$
\begin{equation*}
  \begin{aligned}
   \|u\|_U\le c\sup_{v\in V} \frac{ A(u,v)} {\|v\|_V}
  \le c\sup_{v\in V} \frac{ a(u,v)} {\|v\|_V}+c\sup_{v\in V}\frac{ -(qu,v)} {\|v\|_{V}}
  =c(\|Su \|_{V^*}+\| Tu \|_{V^*}).
  \end{aligned}
\end{equation*}
The rest of the proof is essentially identical with that of Theorem~\ref{t:sourceRL}.
\end{proof}

To relate the variational solution to the strong solution, we require additional
conditions on the potential term $q$ when $\alpha\in (1,3/2]$. In particular,
we first state an ``algebraic'' property of the space $\Hd s \II$, $0<s\leq1$, which is reminiscent
of \cite[Theorem 4.39, pp. 106]{AdamsFournier:2003}.
\begin{lemma}\label{lem:Hsalg}
Let $0<s\leq1, s\neq 1/2$. Then for any $u\in\Hd s\II\cap L^\infty\II$
and $v\in H^s\II\cap L^\infty\II$, the product $uv$ is in $\Hd s\II$.
\end{lemma}
\begin{proof}
Obviously, the product $ uv $ has a bounded $L^2\II$-norm, so
it suffices to show that the $\Hd s\II$-seminorm exists.
First we consider the case $s=1$. The trivial identity $(uv)'=u'v + uv'$ yields
\begin{equation*}
  \|(uv)'\|_{L^2\II}\leq\|v\|_{L^\infty\II}\|u'\|_{L^2\II} + \|u\|_{L^\infty\II}\|v'\|_{L^2\II}<\infty,
\end{equation*}
from which the desired assertion follows immediately.
For $0<s<1$, $s\neq1/2$, we use the definition of the $\Hd s\II$-seminorm
(see, e.g. \cite{LionsMagenes:1968}).
\begin{equation*}
|v|^2_{\Hd s\II} = \int_0^1\int_0^1\frac{|v(x)-v(y)|^2}{|x-y|^{1+2s}}dxdy, ~~~s\not= 1/2.
\end{equation*}
Using the splitting $u(x)v(x)-u(y)v(y)=(u(x)-u(y))v(x)+u(y)(v(x)-v(y))$
and the trivial inequality $(a+b)^2\leq 2(a^2+b^2)$, we deduce that for $s\neq\frac{1}{2}$,
\begin{equation*}
  \begin{aligned}
   |uv|^2_{\Hd s\II} :=   &\int_0^1\int_0^1\frac{(u(x)v(x)-u(y)v(y))^2}{(x-y)^{1+2s}}dxdy\\
    \leq&2\int_0^1\int_0^1\frac{(u(x)-u(y))^2v(x)^2+u(y)^2(v(x)-v(y))^2}{(x-y)^{1+2s}}dxdy\\
    \leq&2\|v\|_{L^\infty\II}^2\int_0^1\int_0^1\frac{(u(x)-u(y))^2}{(x-y)^{1+2s}}dxdy
     + 2\|u\|_{L^\infty\II}^2\int_0^1\int_0^1\frac{(v(x)-v(y))^2}{(x-y)^{1+2s}}dxdy<\infty,
  \end{aligned}
\end{equation*}
where in the last line we have used the condition $u\in \Hd s\II\cap L^\infty\II$ and
$v\in H^s\II\cap L^\infty\II$. Therefore, by the
definition of the space $\Hd s\II$, we have $uv\in \Hd s\II$.
\end{proof}

\begin{remark}
The $L^\infty\II$ condition in Lemma \ref{lem:Hsalg} is only needed for
$s\in (0,1/2)$, since by Sobolev embedding theorem \cite{AdamsFournier:2003},
for $s>1/2$, the space $H^s\II$ embeds continuously into $L^\infty\II$.
Lemma \ref{lem:Hsalg} can also be extended to the case $s=1/2$ as follows:
if $u,v\in L^\infty\II\cap \Hd {1/2}\II$, then $uv\in \Hd{1/2}\II$. The
only change to the proof is that the equivalent norm then reads \cite{LionsMagenes:1968}
\begin{equation*}
  \|v\|^2_{\Hd {1/2}\II}=\int_0^1\int_0^1\frac{(v(x)-v(y))^2}{(x-y)^{2}}dxdy
+ \int_0^1 \bigg( \frac{v(x)^2}{1-x}+\frac{v(x)^2}{x} \bigg)dx.
\end{equation*}
\end{remark}

Now by introducing $0\le \beta<1/2$ so that $\alpha+\beta>3/2$, then we have the following theorem.

\begin{theorem} \label{thm:regcap}
Let $\beta\in[0,1/2)$ and Assumption \ref{ass:caputo} be fulfilled. Suppose that $\langle F,v\rangle=(f,v)$ for some $f\in
\Hd \beta\II$ with $\alpha+\beta>3/2$, and $q\in L^\infty\II\cap H^\beta\II$. Then
the variational solution $u\in U$ of \eqref{varc} is in $\Hd {\alpha/2}\II \cap H^{\alpha+\beta}\II$
and it is a solution of \eqref{strongp}. Further, it satisfies
\begin{equation*}
  \|u\|_{H^{\alpha+\beta}\II}\leq c\|f\|_{\Hd\beta\II}.
\end{equation*}
\end{theorem}
\begin{proof}
Let $u$ be the solution of \eqref{varc}. We consider the problem
\begin{equation}\label{ctf}
-\DDC0{\alpha} w = f-qu.
\end{equation}
By Lemma~\ref{lem:Hsalg}, $qu $ is in $\Hd{\beta}\II$. A strong solution of \eqref{ctf} is given
by \eqref{strongc} with a right hand side $\widetilde{f}=f-qu\in \Hd{\beta}\II$. Since this solution
satisfies the variational problem \eqref{varc}, it coincides with $u$. The regularity $u\in
H^{\alpha+\beta}\II$ is an immediate consequence of Theorem~\ref{l:i0reg}.
\end{proof}

\begin{remark}\label{H12}
Analogous to the proof of Theorem \ref{thm:regcap}, by Theorem 1.4.1.1 of \cite{grisvard},
Remark \ref{rmk:i0}, and a standard bootstrap argument, one
can show that if $f\in \Hd\beta$ and $q\in C^k[0,1]$ with $k>\floor(\beta)+1$, then
the variational solution $u$ is in $H^{\beta+\alpha}\II$. We note that the solution
to problem \eqref{strongp} in the Caputo case can achieve a full regularity, which is
in stark contrast with that for the Riemann-Liouville
case since for the latter, generally the best possible regularity is $H^{\alpha-1+\beta}\II$,
for any $\beta\in[0,1/2)$, due to the inherent presence of the singular term $x^{\alpha-1}$.
\end{remark}

Finally we discuss the adjoint problem in the Caputo case: find $w\in V$ such that
\begin{equation*}
  a(v,w) = \langle v, F\rangle\quad \mbox{ for all } v \in U,
\end{equation*}
for some fixed $F\in U^*$. In the case of $\langle F,v \rangle
= (f,v)$ for some $f\in L^2\II$, the strong form reads
\begin{equation*}
  -\DDR1\alpha w + qw = f,
\end{equation*}
with $w(1)=0$ and $(x^{1-\alpha},w)=0$. By repeating the steps leading to
\eqref{strongc}, we deduce that for $q=0$, the solution $w$ can be expressed as
\begin{equation*}
  w =  c_f(1-x)^{\alpha-1}- {_xI_1^\alpha}f,
\end{equation*}
with the prefactor $c_{f}$ given by
\begin{equation*}
   c_{f}=\frac{(x^{1-\alpha},{_xI_1}^\alpha f)}{(x^{1-\alpha},(1-x)^{\alpha-1})}=\frac{({_0I_x^\alpha} x^{1-\alpha},f)}{B(2-\alpha,\alpha)}=\frac{1}{\Gamma(\alpha)}(x,f),
\end{equation*}
where $B(\cdot,\cdot)$ refers to the Beta function. Clearly, there holds
\begin{equation*}
  |c_{ f}|\leq c|(x^{1-\alpha},{_xI_1^\alpha}f)|\leq c\|f\|_{L^2\II}.
\end{equation*}
Hence $w\in \Hd {\alpha/2}\II \cap\Hdi1{\alpha-1+\beta}$, for any $\beta\in[0,1/2)$. The case
of a general $q$ can be deduced analogously to the proof of Theorem \ref{thm:regrl}. Therefore,
we have the following improved regularity estimate for the adjoint solution $w$ in the Caputo case.
\begin{theorem}\label{thm:regcapadj}
Let Assumption \ref{ass:caputo} hold, and $q\in L^\infty\II$. Then with a right hand side
$\langle F,v\rangle=( f,v)$ for some $f\in L^2\II$, the solution $w$ to
\eqref{varrl} is in $\Hd{\alpha/2}\II\cap\Hdi1{\alpha-1+\beta}$ for any $\beta\in[0,1/2)$, and
further there holds
\begin{equation*}
  \|w\|_{\Hdi1 {\al -1 +\beta}} \le c\|f\|_{L^2\II}.
\end{equation*}
\end{theorem}

\begin{remark}
The adjoint problem for both the Caputo and Riemann-Liouville cases is of Riemann-Liouville type,
albeit with slightly different boundary conditions, and thus shares the same singularity.
\end{remark}

\section{Stability of Finite Element Approximation}\label{sec:fem}
Now we illustrate the application of the variational formulations developed in Section
\ref{sec:stability} in the numerical approximation of the boundary value problem \eqref{strongp}.
We shall analyze the stability of the discrete variational formulations, and derive error
estimates for the discrete approximations.

\subsection{Finite element spaces and their approximation properties}
We introduce a finite element approximation based on an equally spaced partition of the interval
$\II$. We let $h=1/m $ be the mesh size with $m>1$ a positive integer, and consider the nodes
$x_j=jh$, $j=0,\ldots,m$. We then define $U_h$ to be the set of continuous
functions in $U$ which are linear when restricted to the subintervals, $[x_i,x_{i+1}]$,
$i=0,\ldots,m-1$. Analogously, we define $V_h$ to be the set of functions in $V$ which
are linear when restricted to the intervals,  $[x_i,x_{i+1}]$, $i=0,\ldots,m-1$.
Clearly $U_h\subset U$ and $V_h\subset V$ implies that functions in either space are
continuous and vanish at $1$, and functions in $U_h$ vanish at $0$ while $v_h\in V_h$
satisfies the integral constraint $(x^{1-\alpha},v_h)=0$. A basis of the space $V_h$ can be
easily constructed from the standard Lagrangian nodal basis functions.

We first show the approximation
properties of the finite element spaces $U_h$ and $V_h$.
\begin{lemma}\label{fem-interp-U}
Let the mesh ${\mathcal T}_h$ be quasi-uniform.
If $u \in H^\gamma\II \cap \Hd {\al/2}\II$ with $ \alpha/2 \le \gamma \le 2$, then
\begin{equation}\label{approx-Uh}
\inf_{v \in U_h} \| u -v \|_{H^{\al/2}\II} \le ch^{\gamma -\alpha/2} \|u\|_{H^\gamma\II}.
\end{equation}
Further, if $u \in H^\gamma\II\cap V$, then
\begin{equation}\label{approx-Vh}
\inf_{v \in V_h} \|u -v\|_{H^{\alpha/2}\II} \le ch^{\gamma -{\al/2}} \|u\|_{H^\gamma\II}.
\end{equation}
\end{lemma}
\begin{proof}
Let $\Pi_h u \in U_h$ be the standard Lagrange finite element interpolant of  $u$ so that
for any $0 \le s \le 1$, $ \inf_{v \in U_h} \| u -v \|_{H^s\II} \le \| u - \Pi_h u \|_{H^s\II}$.
Then the estimate \eqref{approx-Uh} is an immediate consequence of the local approximation
properties of the interpolant $\Pi_h u$ \cite[Corollary 1.109, pp. 61]{ern-guermond}
for $p=2$ and Sobolev spaces of integer order. The result for the intermediate fractional values
follows from interpolation.

Now we study the approximation properties of the finite element space $V_h$. Given a function
$ u \in V$ we approximate it by a finite element function $\chi$ with $\chi(1)=0$, e.g., the
interpolation function, so that $\|u-\chi\|_{H^{\alpha/2}\II} \le ch^{\gamma-\alpha/2}
\|u\|_{H^\gamma\II}$. Then we introduce the projection operator $P:~H^{\alpha/2}\II\to V$
by $Pu = u - \gamma_u (1-x)$ where $\gamma_u=(u,x^{1-\alpha})
/(x^{1-\alpha}, 1-x)$ ensures that $Pu\in V$. Now in view of the fact that $P\chi\in V_h$ and $P$
is bounded on $H^{\alpha/2}\II$-norm, we get for $u \in V$
\begin{equation*}
   \|u-P\chi\|_{H^{\alpha/2}\II}=\|P(u-\chi)\|_{H^{\alpha/2}\II} \le c\|u-\chi\|_{H^{\alpha/2}\II},
\end{equation*}
and this completes the proof.
\end{proof}

\subsection{Error estimates in the Riemann-Liouville case}
In this case the finite element source problem is to find $u_h\in U_h$ satisfying
\begin{equation}\label{gal:rl}
a(u_h,v) = \langle F,v\rangle,\qquad \hbox{ for all } v\in U_h.
\end{equation}
Here $F\in U^\ast$ is a bounded linear functional on $U$.  In Theorem \ref{thm:femrl} below we establish
error estimates on the finite element approximation $u_h$ when the source term is given by $\langle F,v
\rangle\equiv (f,v) $ with $f\in L^2\II$. To derive the error estimates, we first establish a stability
result for the discrete variational formulation, using the method of Schatz \cite{Schatz-1974}, from which
it follows the existence and uniqueness of a discrete solution $u_h$.

\begin{lemma}\label{lem:disinfsup:riem}
Let Assumption \ref{ass:riem} hold, $f\in L^2\II$, and $q\in L^\infty\II$. Then there is
an $h_0$ such that for all $h\le h_0$ the finite element problem:  Find $u_h\in U_h$ such that
\begin{equation}\label{discp}
a(u_h,v)=(f,v)  \qquad \hbox{ for all } v\in U_h
\end{equation}
has a unique solution.
\end{lemma}
\begin{proof}
The existence and uniqueness of the discrete solution $u_h$ when $q=0$ is an
immediate consequence of Lemma~\ref{coercive}. To show the stability for
$q \not = 0$ we need a discrete {\it inf-sup} condition for the bilinear form
$a(u_h, v)=A(u_h, v) +(q u_h, v)$ defined on the product space $U_h \times U_h$.
To this end, we use the method of Schatz \cite{Schatz-1974}, which
involves two important steps.

The first step uses the fact that $U_h \subset U$ and the bilinear form $ A(\cdot, \cdot)$
defined by \eqref{adef} is coercive in the space $\Hd {\al/2}\II$, i.e., it satisfies
\eqref{coer1} for all $u_h \in U_h$. Hence, the problem of finding $R_h u \in U_h$ such
that  $A(v, R_hu,)=A(v,u)$ for all $v \in U_h$ has a unique solution. Also the solution
$R_h u$ is stable in $\Hd {\al/2}\II$-norm, i.e. $\| R_h u \|_{\Hd {\al/2}\II} \le c
\| u \|_{\Hd {\al/2}\II}$. Next we apply Nitsche's trick to derive an estimate in $L^2\II$-norm
and introduce the adjoint problem: find $w\in \Hd{\alpha/2}\II$ such that $(u-R_hu,v) = A(w,v)$ for all
$v\in \Hd{\alpha/2}\II.$ Then by Theorem \ref{thm:regrl}, 
for any $\beta\in[0,1/2)$, there holds
$\|w\|_{H^{\alpha-1+\beta}\II}\leq c\|u-R_hu\|_{L^2\II}$.
Now by setting $v=u-R_hu$ in the adjoint problem, Galerkin orthogonality and Lemma \ref{fem-interp-U}, we deduce
\begin{equation*}
  \begin{aligned}
   \|u-R_hu\|_{L^2\II}^2 &= A(w,u-R_hu) = A(w-\Pi_h w,u-R_hu)\\
      & \leq c\|u-R_hu\|_{H^{\alpha/2}\II}\|w-\Pi_h w\|_{H^{\alpha/2}\II}\\
      & \leq ch^{\alpha/2-1+\beta}\|u\|_{H^{\alpha/2}\II}\|u-R_hu\|_{L^2\II}.
  \end{aligned}
\end{equation*}
Consequently, we get the following stability estimates for $\beta \in [0,1/2)$
\begin{equation}\label{eqn:Ritz}
  \begin{aligned}
    \|R_h u\|_{\Hd {\al/2}\II} & \le c\| u\|_{\Hd {\al/2}\II},\\
    \| u- R_h u \|_{L^{2}\II} & \le ch^{\al/2 -1 + \beta}  \| u \|_{H^{\al/2}\II}.
  \end{aligned}
\end{equation}

In the second step, we use a kick-back
argument. For any $z_h \in U_h\subset U$ due to the inf-sup condition \eqref{inf-suprl} we have
\begin{equation*}
  \gamma\|z_h\|_U \le \sup_{v \in U} \frac{a(z_h,v)}{ \|v\|_{U}}  \le
   \sup_{v \in U}  \frac{a(z_h,v -R_h v)}{ \|v\|_{U}} +  \sup_{v \in U}  \frac{a(z_h, R_h v)}{ \|v\|_{U}}.
\end{equation*}
Now using the fact that $ q \in L^\infty\II$, the error estimate in \eqref{eqn:Ritz}, and the
regularity pick up with $\beta \in [0,1/2)$ we get the following bound for the first term
\begin{equation*}
  \begin{aligned}
   \sup_{v \in U}  \frac{a(z_h,v -R_h v)}{ \|v\|_{U}}& = \sup_{v \in U}  \frac{(qz_h,v -R_h v)}{ \|v\|_{U}}\\
    & \le c\sup_{v \in U}  \frac{ \|z_h\|_{L^2\II} \|v -R_h v\|_{L^2\II}}{ \|v\|_{U}} \le ch^{\al /2-1 +\beta}  \|z_h\|_U.
  \end{aligned}
\end{equation*}
For the second term we use the stability estimate in \eqref{eqn:Ritz} for the Riesz projection operator $R_h v$
to get
\begin{equation*}
   \sup_{v \in U}  \frac{a(z_h, R_h v)}{ \|v\|_{U}} \le c\sup_{v \in U_h}  \frac{a(z_h, R_h v)}{ \| R_h v\|_{U}}
   = c\sup_{v \in U_h}  \frac{a(z_h, v)}{ \|v\|_{U}}.
\end{equation*}
Now with the choice $h_0^{\al/2-1+\beta} =\frac{1}{2} \gamma/c$ we arrive at the discrete inf-sup condition:
\begin{equation}\label{discr-inf-sup}
  \frac{\gamma}{2} \|z_h\|_U \le \sup_{v \in U_h} \frac{a(z_h,v)}{ \|v\|_{U}} \quad \text{for} \quad h \le h_0.
\end{equation}
Now the desired assertion on existence and uniqueness follows directly.
\end{proof}

Now we can state an error estimate for the discrete approximation $u_h$ in the Riemann-Liouville case.
\begin{theorem}\label{thm:femrl}
Let Assumption \ref{ass:riem} hold, $f\in L^2\II$, and $q\in L^\infty\II$. Then there is an $h_0$ such
that for all $h\le h_0$, the solution $u_h$ to the finite element problem \eqref{discp} satisfies for any $\beta\in[0,1/2)$
\begin{equation*}
   \|u-u_h\|_{L^2\II} + h^{\alpha/2-1+\beta}\|u-u_h\|_{H^{\alpha/2}\II}  \le ch^{\alpha-2+2\beta} \|f\|_{L^2\II}.
\end{equation*}
\end{theorem}
\begin{proof}
By Lemma \ref{lem:disinfsup:riem}, the discrete approximation $u_h$ is well defined.
The error estimate in $H^{\al/2}\II$-norm follows easily from \eqref{discr-inf-sup}
and Galerkin orthogonality. Namely, for any $ v \in U_h$ we have
\begin{equation*}
\begin{aligned}
\frac{\gamma}{2}  \| u-u_h\|_U  & \le \frac{\gamma}{2} \| u- v\|_U +\frac{\gamma}{2}  \|  v -u_h\|_U
  = \frac{\gamma}{2}  \| u-  v \|_U +   \sup_{z \in U_h} \frac{a( v -u_h , z)}{ \|z\|_{U}} \\
 & \le \frac{\gamma}{2}   \| u- v \|_U +  \sup_{z \in U_h} \frac{a( v -u , z)}{ \|z\|_{U}}
  \le c\| u-  v \|_U.
\end{aligned}
\end{equation*}
Then using Lemma \ref{fem-interp-U} and the regularity pickup in Theorem \ref{thm:regrl}  for
$\beta \in [0,1/2)$ we get
\begin{equation*}
\| u-u_h\|_U \le ch^{\al/2 -1 +\beta} \|u\|_{H^{\al -1 +\beta}\II}
\le ch^{\al/2 -1 +\beta} \|f\|_{L^2\II}
\end{equation*}

To derive the $L^2\II$-norm estimate of the error $e=u-u_h$, we apply Nitsche's trick and introduce the
adjoint problem: find $w\in \Hd{\alpha/2}\II$ such that $(v,e) = a(v,w)$ for all $v\in \Hd{\alpha/2}\II.$
Then by Remark \ref{rmk:regadjrl},  for any $\beta\in[0,1/2)$, there holds $\|w\|_{H^{\alpha-1+\beta}\II}
\leq c\|e\|_{L^2\II}$. Now by proceeding as in the proof of Lemma \ref{lem:disinfsup:riem}, we derive
the desired $L^2\II$ estimate. This concludes the proof of the theorem.
\end{proof}

\begin{remark}\label{rmk:feml2sub}
In view of classical interpolation error estimates and the stability estimate in Theorem
\ref{thm:regrl}, the $H^{\alpha/2}\II$-norm estimate of the error for the FEM approximation $u_h$ is
optimal, whereas the $L^2\II$-norm error estimate is suboptimal, which is also confirmed
by numerical experiments. The culprit of the suboptimality is due to the limited
regularity of the adjoint solution, cf. Remark \ref{rmk:regadjrl}.
\end{remark}

The analysis is also valid for the adjoint problem \eqref{adjvarrl}.
\begin{remark}
Let Assumption \ref{ass:riem} hold. Then for any $h<h_0$,
the existence and uniqueness for the solution $w_h$ for the discrete adjoint problem, i.e.,
find $w_h\in U_h$ such that
\begin{equation*}
  a(v,w_h) = (v,f) \quad  \forall v\in U_h,
\end{equation*}
for a given right hand side $f\in L^2\II$, is an immediate consequence the discrete
stability. This together with the stability  estimate in Remark \ref{rmk:regadjrl}
and the preceding Schatz analysis implies
\begin{equation*}
\|w-w_h\|_U \le ch^{\alpha/2-1+\beta} \|w\|_{\Hdi 1{\alpha-1+\beta}}.
\end{equation*}
\end{remark}

\subsection{Error estimates in the Caputo case}
In the case of the Caputo fractional derivative, the finite element source problem is to find $u_h\in
U_h$ satisfying
\begin{equation}\label{gal:c}
a(u_h,v) = \langle F,v \rangle,\qquad \hbox{ for all } v \in V_h.
\end{equation}
Here $F\in V^*$ is a bounded linear functional on $V$. The existence and uniqueness of the
discrete solution $u_h$ even when $q=0$ is not as immediate as the Riemann-Liouville case.
The discrete inf-sup condition in this ``simplest'' case $q=0$ is stated in the following lemma.

\begin{lemma} \label{isDUV}
There is an $h_0>0$ and $c$ independent of $h$ such that for all $h\le h_0$ and $w_h\in U_h$
\begin{equation*}
\|w_h\|_U \le c \sup_{v_h\in V_h } \frac {A(w_h,v_h)} {\|v_h\|_V}.
\end{equation*}
\end{lemma}
\begin{proof}  Let $w_h$ be in $U_h$ and set $v=w_h-\gamma_{w_h}\phi_0$ (recall $\phi_0=(1-x)^{\alpha-1}\in \Hdi1{\alpha/2}$).
The argument leading to \eqref{isUV} implies
$A(w_h,v) \ge c \|w_h\|_{\Hd {\alpha/2} \II}^2$ and consequently, there holds
\begin{equation}\label{phib}
\|v\|_{\Hdi 1 {\alpha/2}} \le c\|w_h\|_{\Hd {\alpha/2}\II}.
\end{equation}
However, since $\phi_0$ is not piecewise linear, $v\notin V_h$. Next we
set $\widetilde v_h=\Pi_h v = w_h - \gamma_{w_h} \Pi_h \phi_0$, where
$\Pi_h$ denotes the Lagrangian interpolation operator.
We clearly have
\begin{equation}\label{ff1}
  \begin{aligned}
    \|v-\widetilde v_h\|_{H^{\alpha/2}\II} & = |\gamma_{w_h}|\|(I-\Pi_h)\phi_0\|_{H^{\alpha/2}\II} \\
    & \le c|\gamma_{w_h}| h^{\alpha-1+\beta-\alpha/2} \|\phi_0 \|_{\Hdi 1 {\alpha-1+\beta}}
  \end{aligned}
\end{equation}
for any $\beta\in [0,1/2)$ with $\alpha/2-1+\beta>0$ by Lemma \ref{phi-0}.
Now in order to obtain a function $v_h$ in the space $V_h$, we define
$v_h = \widetilde v_h - \eta (1-x)$, with $\eta$ given by
$\eta=\frac {(x^{1-\alpha},\widetilde v_h)} {(x^{1-\alpha},1-x)}.$
Meanwhile, there holds
\begin{equation}\label{ff2}
  \|\widetilde v_h-v_h\|_{H^{\alpha/2}\II} = |\eta|\|1-x\|_{H^{\alpha/2}\II}.
\end{equation}
Next we bound the numerator $(x^{1-\alpha},\widetilde{v}_h)$
in the definition of $\eta$. Since the function $v$ is in $V$, we deduce
\begin{equation}
|(x^{1-\alpha},\widetilde v_h)|=|(x^{1-\alpha},\widetilde v_h-v)|
=|\gamma_{w_h}| |(x^{1-\alpha}, (I-\Pi_h) \phi_0)|.\label{ff4}
\end{equation}
Now we estimate the term $|(x^{1-\alpha}, (I-\Pi_h) \phi_0)|$.
By the concavity of the function $\phi_0$, there holds $\phi_0(x)-\Pi_h\phi_0(x) \ge 0$
and hence $|(x^{1-\alpha}, (I-\Pi_h) \phi_0)|=(x^{1-\alpha}, (I-\Pi_h) \phi_0)$.
Further for $x\in [x_i,x_{i+1}]$, by the mean value theorem
\begin{equation}\label{interpe}
\phi_0(x)-\Pi_h\phi_0(x) = \frac {-\phi_0^{\prime\prime}(\zeta_x)}2  (x_{i+1}-x)
(x-x_i),\quad \zeta_x\in (x_i,x_{i+1}).
\end{equation}
We break the integral $(x^{1-\alpha},(I-\Pi_h)\phi_0)$ into two pieces:
\begin{equation*}
(x^{1-\alpha}, (I-\Pi_h) \phi_0)=
\int_0^{1/2} x^{1-\alpha}(I-\Pi_h) \phi_0\, dx +
\int_{1/2}^1 x^{1-\alpha}(I-\Pi_h) \phi_0\, dx\equiv J_1+J_2.
\end{equation*}
As to the first integral $J_1$, by applying \eqref{interpe}, we find
\begin{equation*}
J_1\le c\|\phi_0^{\prime\prime}\|_{L^\infty(0,1/2)}  h^2
 \int_0^{1/2} x^{1-\alpha} dx \le ch^2.
\end{equation*}
We break up the second integral $J_2$ further, i.e.,
\begin{equation*}
   \begin{aligned}
   J_2 &= \int_{1/2}^{1-2h} x^{1-\alpha}(I-\Pi_h) \phi_0\, dx +\int_{1-2h}^1 x^{1-\alpha}(I-\Pi_h) \phi_0\, dx\\
   &\le \frac 12 \sum_{i=\floor(m/2)}^{m-3} \|x^{1-\alpha}\|_{L^\infty(1/2,1)}
     (-\phi_0^{\prime\prime}(x_{i+1})) \int_{x_i}^{x_{i+1}}  (x_{i+1}-x)(x-x_i)\, dx +\int_{1-2h}^1 x^{1-\alpha}\phi_0\, dx\\
   &\le ch^2\sum_{i=\floor(m/2)}^{m-3} h (-\phi_0^{\prime\prime}(x_{i+1}))
     + ch^\alpha \le ch^2 \int_{x_{\floor(m/2)+1}}^{1-h} (-\phi_0^{\prime\prime}(x)) dx + ch^\alpha \le ch^\alpha.
   \end{aligned}
\end{equation*}
The above two estimates together imply
\begin{equation}\label{ff3}
(x^{1-\alpha}, (I-\Pi_h) \phi_0)\le ch^\alpha.
\end{equation}

Now, combining \eqref{ff1},\eqref{ff2}, \eqref{ff4}  and \eqref{ff3} and
using $|\gamma_{w_h}| \le c\|w_h\|_U$ gives
\begin{equation*}
  \|v-v_h\|_{H^{\alpha/2}\II} \le  ch^{\alpha/2-1+\beta} \|w_h\|_U.
\end{equation*}
This together with \eqref{phib} and the triangle inequality also implies that
$\|v_h\|_{H^{\alpha/2}\II}\le c\|w_h\|_U$. Thus,
\begin{equation*}
  \begin{aligned}
    \|w_h\|_U \le \frac {cA(w_h,v)} {\| w_h \|_U} &
         \le \frac { cA(w_h,v_h)} {\| w_h \|_U}+\frac {cA(w_h,v-v_h)} {\|w_h\|_U}\\
      &\le \frac { cA(w_h,v_h)} {\| w_h \|_U}+c_1 h^{\alpha/2-1+\beta} \|w_h\|_U.
  \end{aligned}
\end{equation*}
Taking $h_0$ so that $c_1 h_0^{\alpha/2-1+\beta}\le 1/2$ we get
$$\|w_h\|_U \le \frac {c A(w_h,v_h)} {\| w_h \|_U}
\le \frac {cA(w_h,v_h)} {\|v_h \|_{V}}
$$
from which the lemma immediately follows.
\end{proof}

Now we can show the {\it inf-sup} condition for the bilinear form $a(v,\phi)$ with a general
$q$ defined on $U_h \times V_h$.
\begin{lemma}\label{lem:disinfsup:Cap}
Let Assumption \ref{ass:caputo} be fulfilled. Then there is an $h_0>0$
such that for all $h\le h_0$
\begin{equation*}
\forall w_h\in U_h\quad\quad \|w_h\|_U \le c \sup_{v_h\in V_h } \frac {a(w_h,v_h)} {\|v_h\|_V}.
\end{equation*}
\end{lemma}
\begin{proof}
The proof is essentially identical with that for Lemma \ref{lem:disinfsup:riem}, i.e., with Lemma \ref{isDUV}
in place of Lemma \ref{coercive}, together with the Schatz argument \cite{Schatz-1974}, and hence omitted.
\end{proof}

Finally we state an error estimate on the finite element approximation $u_h$ when the source term is given
by $\langle F,v \rangle \equiv (f,v) $ with $f\in \Hd\beta\II$ with $\beta\in[0,1/2)$ such that $\alpha+\beta>3/2$.
\begin{theorem}\label{thm:femcap}
Let Assumption \ref{ass:caputo} hold, $f\in \Hd\beta\II$, and $q\in L^\infty\II\cap \Hd\beta\II$
for some $\beta\in[0,1/2)$ such that $\alpha+\beta>3/2$. Then there is an $h_0$ such that
for $h\le h_0$, the finite element problem:  Find $u_h\in U_h$ satisfying
\begin{equation}\label{disccap}
a(u_h,v)=(f,v)  ,\qquad \hbox{ for all } v\in V_h,
\end{equation}
has a unique solution. Moreover, the solution $u_h$ satisfies that for any $\delta\in[0,1/2)$
\begin{equation*}
  \begin{aligned}
    \|u-u_h\|_{\Hd{\alpha/2}\II} & \le ch^{\min(\alpha+\beta,2)-\alpha/2} \|f\|_{\Hd{\beta}\II},\\
    \|u-u_h\|_{L^2\II} &\leq ch^{\min(\alpha+\beta,2)-1+\delta}\|f\|_{\Hd{\beta}\II}.
  \end{aligned}
\end{equation*}
\end{theorem}
\begin{proof}
According to Theorem \ref{thm:regcap}, the solution $u$ to \eqref{disccap}
is in $H^{\alpha+\beta}\II\cap \Hd {\alpha/2}\II$. By Lemma \ref{lem:disinfsup:Cap}, the discrete solution $u_h$
is well defined. The $\Hd {\alpha/2}\II$-norm estimate of the error $e=u_h-u$
follows from C\'{e}a's lemma (cf. e.g. \cite[pp. 96, Lemma 2.28]{ern-guermond}).
To derive the $L^2\II$ estimate, we appeal again to the Nitsche's trick and introduce
the adjoint problem: find $w\in V$ such that
\begin{equation*}
  a(v,w)=(v,e)\quad \mbox{ for all } v \in U.
\end{equation*}
By Theorem \ref{thm:regcapadj}, the solution $w$ lies in the space $\Hdi1{\alpha-1+\delta}$
for any $\delta\in[0,1/2)$, and satisfies the a priori estimate $\|w\|_{H^{\alpha-1+\delta}\II}\leq c\|e\|_{L^2\II}$.
The rest of the proof is identical with that of Theorem \ref{thm:femrl}.
\end{proof}

\begin{remark}
Like in the case of Riemann-Liouville fractional derivative, the $L^2\II$-norm of the
error is suboptimal due to the insufficient regularity of the adjoint solution, cf. Theorem \ref{thm:regcapadj}.
\end{remark}

\section{Numerical experiments and discussions}\label{sec:num}
In this part, we present some numerical experiments to illustrate the theory, i.e.,
the optimality of the finite element error estimates, and the difference between
the regularity estimates for the Riemann-Liouville and the Caputo fractional derivatives.

We consider the following three examples, with the source term $f$ of different regularity.
\begin{itemize}
  \item[(a)] The source term $f(x)=x(1-x)$ belongs to $\Hd {1+\delta}\II$ for any $\delta\in [0,1/2)$.
  \item[(b)] The source term $f(x)=1$ belongs to $\Hd {\delta}\II$ for any $\delta\in [0,1/2)$.
  \item[(c)] The source term $f(x)=x^{-1/4}$ belongs to the space $\Hd {\delta}\II$ for any $\delta\in [0,1/4)$.
\end{itemize}

The computations were performed on uniform meshes of mesh sizes $h=1/(2^k\times10)$, $k=1,2,\ldots,7$.
In the examples, we set the potential term $q$ to zero, so that the exact solution
can be computed directly using the solution representations \eqref{strongrl} and \eqref{strongc}. For
each example, we consider three different $\alpha$ values, i.e., $7/4$, $3/2$ and $4/3$, and
present the $L^2\II$-norm and $H^{\alpha/2}\II$-norm of the error $e=u-u_h$.

\subsection{Numerical results for example (a)} The source term $f$ is smooth, and the exact solution $u(x)$ is given by
\begin{equation*}
  u(x) =\left\{\begin{aligned}
    \tfrac{1}{\Gamma(\alpha+2)}(x^{\alpha-1}-x^{\alpha+1})-\tfrac{2}{\Gamma(\alpha+3)}(x^{\alpha-1}-x^{\alpha+2}), & \quad \text{Riemann-Liouville case},  \\  \tfrac{1}{\Gamma(\alpha+2)}(x-x^{\alpha+1})-\tfrac{2}{\Gamma(\alpha+3)}(x-x^{\alpha+2}), & \quad \text{Caputo case}.
  \end{aligned}\right.
\end{equation*}
According to Theorems \ref{thm:regrl} and \ref{thm:regcap}, the solution $u$ belongs to $\Hdi0{\alpha-1+\beta}$ and
$H^{\alpha+1+\beta}\II$ for any $\beta\in [0,1/2)$ for the Riemann-Liouville and Caputo case, respectively. In particular,
in the Caputo case, the solution $u$ belongs to $H^2\II$. Hence,
in the Riemann-Liouville case, the finite element approximations $u_h$ would converge at a rate $O(h^{\alpha/2-1/2})$
and $O(h^{2\alpha-1})$ in $H^{\alpha/2}\II$-norm and $L^2\II$-norm, respectively, whereas that for the Caputo case would be
at a rate $O(h^{2-\alpha/2})$ and $O(h^{3/2})$, respectively. Meanwhile, we note that despite
the smoothness of the source term $f$, the Riemann-Liouville solution is nonsmooth, due the presence of the term
$x^{\alpha-1}$, and hence the numerical scheme can only converge slowly in the $H^{\alpha/2}\II$-norm.
The numerical results are shown in Tables \ref{tab:exam1:riem}
and \ref{tab:exam1:cap}, where the number in the bracket under the column \texttt{rate} is the theoretical convergence
rate derived in Theorems \ref{thm:femrl} and \ref{thm:femcap}. The results indicate that the $H^{\alpha/2}\II$
estimates are fully confirmed; however, the $L^2\II$ estimates are only suboptimal. The actual convergence rate in
$L^2\II$-norm is one half order higher than the theoretical prediction, for both fractional derivatives, which agrees with
the conjecture in Remark \ref{rmk:feml2sub}. Intuitively, this might be explained by the structure of the adjoint problem:
The adjoint solution contains the singular term $(1-x)^{\alpha-1}$, with a coefficient $({_xI_1^\alpha} e)(0)=\frac{1}{\Gamma(\alpha)}
\int_0^1x^{\alpha-1}e(x)dx$ for the Riemann-Liouville case (respectively $\frac{1}{\Gamma(\alpha)}(x,e)$ for the Caputo case);
the error function $e$ has large oscillations mainly around the origin, which is however compensated by
the weight $x^{\alpha-1}$ in the Riemann-Liouville case (the weight is $x$ in the Caputo case) in the integral, and thus the
coefficient is much smaller than the apparent $L^2\II$-norm of the error. This can be numerically confirmed: the coefficient
decays much faster to zero than the $L^2\II$-norm of the error, cf. Table \ref{tab:exam1:coef}, and hence in the adjoint
solution representation, the crucial term $x^{\alpha-1}$ plays a less significant role than that in the primal problem.
Nonetheless, this higher-order convergence rate awaits mathematical justification.

\begin{table}[h!]{\small
  \caption{Numerical results for example (a) with a Riemann-Liouville fractional derivative, $f(x)=x(1-x)$, discretized
  on a uniform mesh of mesh size $h=1/(2^k\times10)$.}\label{tab:exam1:riem}
  \begin{tabular}{|c|c|c|c|c|c|c|c|c|c|}
     \hline
      $\al$ & $k$ & $1$ & $2$ & $3$ &$4$ & $5$ & $6$ & $7$ & rate\\
     \hline
     $7/4$ & $L^2$-norm     & 8.53e-3 & 6.43e-3 & 4.91e-3 & 3.77e-3 & 2.90e-3 & 2.23e-3 & 1.72e-3 & 1.25 (0.75)\\
     \cline{2-10}
     & $H^{\alpha/2}$-norm  & 1.70e-4 & 7.05e-5 & 2.95e-5 & 1.24e-5 & 5.21e-6 & 2.19e-6 & 9.21e-7 & 0.39 (0.38)\\
     \hline
     $3/2$ & $L^2$-norm     & 1.08e-3 & 5.40e-4 & 2.70e-4 & 1.35e-4 & 6.74e-5 & 3.37e-5 & 1.68e-5 & 1.00 (0.50)\\
     \cline{2-10}
     & $H^{\alpha/2}$-norm  & 2.85e-2 & 2.39e-2 & 2.00e-2 & 1.68e-2 & 1.41e-2 & 1.18e-2 & 9.82e-3 & 0.26 (0.25)\\
     \hline
     $4/3$ & $L^2$-norm     & 3.50e-3 & 1.96e-3 & 1.10e-3 & 6.16e-4 & 3.46e-4 & 1.94e-4 & 1.09e-4 & 0.83 (0.33)\\
     \cline{2-10}
     & $H^{\alpha/2}$-norm  & 5.40e-2 & 4.79e-2 & 4.25e-2 & 3.76e-2 & 3.33e-2 & 2.93e-2 & 2.58e-2 & 0.18 (0.17)\\
     \hline
     \end{tabular}}
\end{table}

\begin{table}[h!]{\small
  \caption{Numerical results for example (a) with a Caputo fractional derivative, $f(x)=x(1-x)$, discretized
  on a uniform mesh of mesh size $h=1/(2^k\times10)$.}\label{tab:exam1:cap}
  \begin{tabular}{|c|c|c|c|c|c|c|c|c|c|}
     \hline
      $\al$ & $k$ & $1$ & $2$ & $3$ &$4$ & $5$ & $6$ & $7$ & rate\\
     \hline
     $7/4$     & $L^2$-norm & 2.45e-5 & 5.98e-6 & 1.48e-6 & 3.72e-7 & 9.38e-8 & 2.37e-8 & 6.00e-9 & 2.00 (1.50)\\
     \cline{2-10}
     & $H^{\alpha/2}$-norm  & 1.50e-3 & 6.88e-4 & 3.15e-4 & 1.44e-4 & 6.62e-5 & 3.04e-5 & 1.39e-5 & 1.13 (1.13)\\
     \hline
     $3/2$     & $L^2$-norm & 4.93e-5 & 1.25e-5 & 3.14e-6 & 7.92e-7 & 1.99e-7 & 4.99e-8 & 1.25e-8 & 1.99 (1.50)\\
     \cline{2-10}
     & $H^{\alpha/2}$-norm  & 8.84e-4 & 3.69e-4 & 1.54e-4 & 6.48e-5 & 2.72e-5 & 1.14e-5 & 4.81e-6 & 1.25 (1.25)\\
     \hline
     $4/3$    & $L^2$-norm & 7.40e-5 & 1.85e-5 & 4.62e-6 & 1.16e-6 & 2.89e-7 & 7.24e-8 & 1.81e-8 & 2.00 (1.50)\\
     \cline{2-10}
     & $H^{\alpha/2}$-norm  & 6.24e-4 & 2.43e-4 & 9.54e-5 & 3.77e-5 & 1.49e-5 & 5.91e-6 & 2.35e-6 & 1.34 (1.33)\\
     \hline
     \end{tabular}}
\end{table}

\begin{table}[h!]{\small
   \caption{The coefficient $({_xI_1^\alpha e})(0)$ for Riemann-Liouville case (respectively $\frac{1}{\Gamma(\alpha)}(x,e)$ for the Caputo case)
   in the adjoint solution representation for example (a), with $\alpha=3/2$, discretized on a uniform mesh of mesh size $h=1/(2^k\times10)$.}\label{tab:exam1:coef}
  \begin{tabular}{|c|c|c|c|c|c|c|c|c|}
     \hline
      $k$              & $1$ & $2$ & $3$ &$4$ & $5$ & $6$ & $7$ \\
     \hline
     R.-L.   &3.65e-3 & 5.09e-4 & 6.98e-5 & 9.46e-6 & 1.27e-6 & 1.70e-7 & 2.26e-8\\
     Caputo  &1.76e-3 & 2.24e-4 & 2.85e-5 & 3.59e-6 & 4.53e-7 & 5.69e-8 & 7.13e-9\\
     \hline
     \end{tabular}}
\end{table}

\subsection{Numerical results for example (b)}
In this example, the source term $f$ is very smooth, but does not satisfy the
zero boundary condition. Hence it belongs to $\Hd \beta\II$ for $\beta\in[0,1/2)$.
The exact solution $u$ is given by
\begin{equation*}
  u(x) = \left\{\begin{aligned}
     c_\alpha(x^{\alpha-1}-x^{\alpha}),\quad& \text{Riemann-Liouville case},\\
     c_\alpha(x-x^{\alpha}),\quad& \text{Caputo case},
  \end{aligned}\right.
\end{equation*}
with the constant $c_\alpha=\frac{1}{\Gamma(\alpha+1)}$. The numerical results are shown in Tables
\ref{tab:exam2:riem} and \ref{tab:exam2:cap} for the Riemann-Liouville and Caputo case, respectively.
In the Riemann-Liouville case, the convergence rates
are identical with that for example (a). This is attributed to the fact that the regularity of the
solution $u$ generally cannot go beyond $\Hdi0 {\alpha-1+\beta}$ for $\beta\in[0,1/2)$, i.e.,
 which is independent of the regularity of the source term $f$. In contrast, in
the Caputo case, for $\alpha\geq 3/2$, we observe identical convergence rates as example (a),
whereas for $\alpha=4/3$, the convergence is slower due to limited smoothing induced by the fractional
differential operator. Like before, for either fractional derivative, the empirical convergence in
$L^2\II$-norm is better than the theoretical prediction by one-half order.

\begin{table}[h!]{\small
  \caption{Numerical results for example (b) with a Riemann-Liouville fractional derivative,
  $f(x)=1$, discretized on a uniform mesh of mesh size $h=1/(2^k\times10)$.}\label{tab:exam2:riem}
  \begin{tabular}{|c|c|c|c|c|c|c|c|c|c|}
     \hline
      $\al$ & $k$ & $1$ & $2$ & $3$ &$4$ & $5$ & $6$ & $7$ & rate\\     
     \hline
     $7/4$     & $L^2$-norm & 1.07e-3 & 4.31e-4 & 1.77e-4 & 7.37e-5 & 3.08e-5 & 1.29e-5 & 5.43e-6 & 1.27 (0.75)\\
     \cline{2-10}
     & $H^{\alpha/2}$-norm  & 5.26e-2 & 3.90e-2 & 2.94e-2 & 2.24e-2 & 1.72e-2 & 1.32e-2 & 1.01e-2 & 0.40 (0.38)\\
     \hline
     $3/2$     & $L^2$-norm & 6.44e-3 & 3.18e-3 & 1.58e-3 & 7.89e-4 & 3.94e-4 & 1.97e-4 & 9.84e-5 & 1.01 (0.50)\\
     \cline{2-10}
     & $H^{\alpha/2}$-norm  & 1.69e-1 & 1.40e-1 & 1.17e-1 & 9.82e-2 & 8.22e-2 & 6.87e-2 & 5.73e-2 & 0.26 (0.25)\\
     \hline
     $4/3$     & $L^2$-norm & 2.05e-2 & 1.15e-2 & 6.42e-3 & 3.60e-3 & 2.02e-3 & 1.13e-3 & 6.35e-4 & 0.84 (0.33)\\
     \cline{2-10}
     & $H^{\alpha/2}$-norm  & 3.17e-1 & 2.80e-1 & 2.48e-1 & 2.20e-1 & 1.94e-1 & 1.71e-1 & 1.50e-1 & 0.18 (0.17)\\
     \hline
     \end{tabular}}
\end{table}

\begin{table}[h!]{\small
  \caption{Numerical results for example (b) with a Caputo fractional derivative, $f(x)=1$, discretized on a
    uniform mesh of mesh size $h=1/(2^k\times10)$.}\label{tab:exam2:cap}
  \begin{tabular}{|c|c|c|c|c|c|c|c|c|c|}
     \hline
      $\al$ & $k$ & $1$ & $2$ & $3$ &$4$ & $5$ & $6$ & $7$ & rate\\
     \hline
     $7/4$ & $L^2$-norm     & 1.74e-4 & 4.21e-5 & 1.03e-5 & 2.51e-6 & 6.16e-7 & 1.51e-7 & 3.74e-8 & 2.00 (1.50)\\
     \cline{2-10}
     & $H^{\alpha/2}$-norm  & 8.14e-3 & 3.74e-3 & 1.72e-3 & 7.91e-4 & 3.63e-4 & 1.67e-4 & 7.65e-5 & 1.12 (1.13)\\
     \hline
     $3/2$ & $L^2$-norm     & 1.88e-4 & 4.84e-5 & 1.24e-5 & 3.17e-6 & 8.12e-7 & 2.07e-7 & 5.29e-8 & 1.97 (1.50)\\
     \cline{2-10}
     & $H^{\alpha/2}$-norm  & 4.81e-3 & 2.12e-3 & 9.33e-4 & 4.08e-4 & 1.78e-4 & 7.76e-5 & 3.37e-5 & 1.20 (1.25)\\
     \hline
     $4/3$ &     $L^2$-norm & 2.48e-4 & 6.99e-5 & 1.97e-5 & 5.53e-6 & 1.55e-6 & 4.36e-7 & 1.22e-7 & 1.83 (1.33)\\
     \cline{2-10}
     & $H^{\alpha/2}$-norm  & 3.44e-3 & 1.55e-3 & 6.96e-4 & 3.12e-4 & 1.40e-4 & 6.26e-5 & 2.80e-5 & 1.16 (1.17)\\
     \hline
     \end{tabular}}
\end{table}

\subsection{Numerical results for example (c)}
In this example, the source term $f$ is singular at the origin, and it belongs to the space
$\Hd \beta\II$ for any $\beta\in[0,1/4)$. The exact solution $u$ is given by
\begin{equation*}
  u(x) = \left\{\begin{aligned}
     c_\alpha(x^{\alpha-1}-x^{\alpha-1/4}),\quad& \text{Riemann-Liouville case},\\
     c_\alpha(x-x^{\alpha-1/4}),\quad& \text{Caputo case}.\\
  \end{aligned}\right.
\end{equation*}
with the constant $c_\alpha=\frac{\Gamma(3/4)}{\Gamma(\alpha+3/4)}$. The numerical results are
shown in Tables \ref{tab:exam3:riem} and \ref{tab:exam3:cap}. In the Riemann-Liouville case, the same
convergence rates are observed, cf. Tables \ref{tab:exam1:riem} and \ref{tab:exam2:riem}, concurring
with Remark \ref{rmk:regrl} and earlier observations. In the Caputo case,
due to the lower regularity of the source term $f$, the numerical solution $u_h$ converges slower as the fractional
order $\alpha$ approaches $1$, but the empirical convergence behavior still agrees well with the theoretical prediction.

\begin{table}[h!]{\small
  \caption{Numerical results for example (c) with a Riemann-Liouville fractional derivative, $f(x)=x^{-1/4}$,
  discretized on a uniform mesh of mesh size $h=1/(2^k\times10)$.}\label{tab:exam3:riem}
  \begin{tabular}{|c|c|c|c|c|c|c|c|c|c|}
     \hline
      $\al$ & $k$ & $1$ & $2$ & $3$ &$4$ & $5$ & $6$ & $7$ & rate\\
     \hline
     $7/4$ & $L^2$-norm     & 1.65e-3 & 6.61e-4 & 2.69e-4 & 1.11e-4 & 4.62e-5 & 1.93e-5 & 8.09e-6 & 1.28 (0.75)\\
     \cline{2-10}
     & $H^{\alpha/2}$-norm  & 8.07e-2 & 5.94e-2 & 4.45e-2 & 3.38e-2 & 2.57e-2 & 1.97e-2 & 1.51e-2 & 0.40 (0.38)\\
     \hline
     $3/2$ & $L^2$-norm     & 9.31e-3 & 4.60e-3 & 2.29e-3 & 1.14e-3 & 5.68e-4 & 2.83e-4 & 1.41e-4 & 1.01 (0.50)\\
     \cline{2-10}
     & $H^{\alpha/2}$-norm  & 2.44e-1 & 2.03e-1 & 1.69e-1 & 1.41e-1 & 1.18e-1 & 9.89e-2 & 8.25e-2 & 0.26 (0.25)\\
     \hline
     $4/3$   & $L^2$-norm   & 2.88e-2 & 1.61e-2 & 9.02e-3 & 5.06e-3 & 2.84e-3 & 1.59e-3 & 8.93e-4 & 0.84 (0.33)\\
     \cline{2-10}
     & $H^{\alpha/2}$-norm  & 4.44e-1 & 3.94e-1 & 3.49e-1 & 3.09e-1 & 2.73e-1 & 2.41e-1 & 2.11e-1 & 0.18 (0.17)\\
     \hline
     \end{tabular}}
\end{table}

\begin{table}[h!]{\small
  \caption{Numerical results for example (c) with a Caputo fractional derivative, $f(x)=x^{-1/4}$, discretized
  on a uniform mesh of mesh size $h=1/(2^k\times10)$.}\label{tab:exam3:cap}
  \begin{tabular}{|c|c|c|c|c|c|c|c|c|c|}
     \hline
      $\al$ & $k$ & $1$ & $2$ & $3$ &$4$ & $5$ & $6$ & $7$ & rate\\
     \hline
     $7/4$ & $L^2$-norm     & 3.04e-4 & 7.67e-5 & 1.93e-5 & 4.88e-6 & 1.23e-6 & 3.11e-7 & 7.86e-8 & 1.99 (1.50)\\
     \cline{2-10}
     & $H^{\alpha/2}$-norm  & 1.21e-2 & 5.85e-3 & 2.82e-3 & 1.35e-3 & 6.46e-4 & 3.08e-4 & 1.46e-4 & 1.07 (1.13)\\
     \hline
     $3/2$ & $L^2$-norm     & 3.93e-4 & 1.09e-4 & 3.06e-5 & 8.69e-6 & 2.49e-6 & 7.21e-7 & 2.10e-7 & 1.81 (1.25)\\
     \cline{2-10}
     & $H^{\alpha/2}$-norm  & 6.84e-3 & 3.48e-3 & 1.75e-3 & 8.82e-4 & 4.43e-4 & 2.22e-4 & 1.11e-4 & 0.99 (1.00)\\
     \hline
     $4/3$ & $L^2$-norm     & 4.31e-4 & 1.18e-4 & 3.33e-5 & 9.83e-6 & 3.01e-6 & 9.47e-7 & 3.05e-7 & 1.70 (1.08)\\
     \cline{2-10}
     & $H^{\alpha/2}$-norm  & 2.60e-3 & 1.43e-3 & 7.72e-4 & 4.13e-4 & 2.20e-4 & 1.17e-4 & 6.19e-5 & 0.90 (0.92)\\
     \hline
     \end{tabular}}
\end{table}

\section{Conclusions}

In this work we have developed variational formulations for boundary value problems involving
either Riemann-Liouville or Caputo fractional derivatives of order $\alpha\in(1,2)$. The stability
of the variational formulations, and the Sobolev regularity of the variational solutions was established.
Moreover, the finite element discretization of the scheme was developed, and convergence rates
in $\Hd {\alpha/2}\II$- and $L^2\II$-norms were established. The error estimates in $\Hd {\alpha/2}\II$-norm
were fully supported by the numerical experiments, whereas the $L^2\II$-estimates remain one-half order lower
than the empirical convergence rates, which requires further investigations. There are several avenues
for future works. First, in the Riemann-Liouville case, the solution generally contains an inherent
singularity of the form $x^{\alpha-1}$, and numerically on a uniform mesh we have observed that the numerical
solutions suffer from pronounced spurious oscillations around the origin, especially for $\alpha$ close to
unity. This necessitates the use of an appropriate adaptively refined mesh or a locally enriched solution space.
Second, one important application of the variational formulations developed herein is fractional Sturm-Liouville
problems \cite{JinRundell:2012}, which have shown some unusual features. Third, it is natural to extend the
analysis to the parabolic counterpart, i.e., space fractional diffusion problems.

\section*{Acknowledgements}
The research of B. Jin has been supported by US NSF Grant DMS-1319052, R. Lazarov was supported in parts
by US NSF Grant DMS-1016525 and J. Pasciak
has been supported by NSF Grant DMS-1216551. The work of all authors has been supported also
by Award No. KUS-C1-016-04, made by King Abdullah University of Science and Technology (KAUST).

\bibliographystyle{abbrv}
\bibliography{frac}

\begin{thebibliography}{10}

\bibitem{AdamsFournier:2003}
R.~A. Adams and J.~J.~F. Fournier.
\newblock {\em Sobolev {S}paces}.
\newblock Elsevier/Academic Press, Amsterdam, second edition, 2003.

\bibitem{BensonWheatcraftMeerschaert:2000}
D.~A. Benson, S.~W. Wheatcraft, and M.~M. Meerschaert.
\newblock The fractional-order governing equation of {L}\'{e}vy motion.
\newblock {\em Water Resour. Res.}, 36(6):1413--1424, 2000.

\bibitem{ern-guermond}
A.~Ern and J.-L. Guermond.
\newblock {\em Theory and {P}ractice of {F}inite {E}lements}, volume 159 of
  {\em Applied Mathematical Sciences}.
\newblock Springer-Verlag, New York, 2004.

\bibitem{ErvinHeuerRoop:2007}
V.~Ervin, N.~Heuer, and J.~Roop.
\newblock Numerical approximation of a time dependent, nonlinear,
  space-fractional diffusion equation.
\newblock {\em SIAM J. Numer. Anal.}, 45(2):572--591, 2007.

\bibitem{ErvinRoop:2006}
V.~Ervin and J.~Roop.
\newblock Variational formulation for the stationary fractional advection
  dispersion equation.
\newblock {\em Numer. Methods Partial Diff. Eq.}, 22(3):558--576, 2006.

\bibitem{ErvinRoop:2007}
V.~J. Ervin and J.~P. Roop.
\newblock Variational solution of fractional advection dispersion equations on
  bounded domains in {$\Bbb R^d$}.
\newblock {\em Numer. Methods Partial Differential Equations}, 23(2):256--281,
  2007.

\bibitem{grisvard}
P.~Grisvard.
\newblock {\em Elliptic {P}roblems in {N}onsmooth {D}omains}.
\newblock Pitman, Boston, MA, 1985.

\bibitem{JinRundell:2012}
B.~Jin and W.~Rundell.
\newblock An inverse {S}turm-{L}iouville problem with a fractional derivative.
\newblock {\em J. Comput. Phys.}, 231(14):4954--4966, 2012.

\bibitem{KilbasSrivastavaTrujillo:2006}
A.~Kilbas, H.~Srivastava, and J.~Trujillo.
\newblock {\em Theory and {A}pplications of {F}ractional {D}ifferential
  {E}quations}.
\newblock Elsevier, Amsterdam, 2006.

\bibitem{LionsMagenes:1968}
J.-L. Lions and E.~Magenes.
\newblock {\em Probl\`emes aux limites non homog\`enes et applications. {V}ol.
  1}.
\newblock Travaux et Recherches Math\'ematiques, No. 17. Dunod, Paris, 1968.

\bibitem{PangSun:2012}
H.-K. Pang and H.-W. Sun.
\newblock Multigrid method for fractional diffusion equations.
\newblock {\em J. Comput. Phys.}, 231(2):693--703, 2012.

\bibitem{Podlubny_book}
I.~Podlubny.
\newblock {\em Fractional {D}ifferential {E}quations}.
\newblock Academic Press, San Diego, CA, 1999.

\bibitem{PodlubnyChechkin:2009}
I.~Podlubny, A.~Chechkin, T.~Skovranek, Y.~Chen, and B.~M. Vinagre~Jara.
\newblock Matrix approach to discrete fractional calculus. {II}. {P}artial
  fractional differential equations.
\newblock {\em J. Comput. Phys.}, 228(8):3137--3153, 2009.

\bibitem{SamkoKilbasMarichev:1993}
S.~G. Samko, A.~A. Kilbas, and O.~I. Marichev.
\newblock {\em Fractional {I}ntegrals and {D}erivatives}.
\newblock Gordon and Breach Science Publishers, Yverdon, 1993.

\bibitem{Schatz-1974}
A.~H. Schatz.
\newblock An observation concerning {R}itz-{G}alerkin methods with indefinite
  bilinear forms.
\newblock {\em Math. Comp.}, 28:959--962, 1974.

\bibitem{Shkhanukov:1996}
M.~K. Shkhanukov.
\newblock On the convergence of difference schemes for differential equations
  with a fractional derivative.
\newblock {\em Dokl. Akad. Nauk}, 348(6):746--748, 1996.

\bibitem{Sousa:2009}
E.~Sousa.
\newblock Finite difference approximations for a fractional advection diffusion
  problem.
\newblock {\em J. Comput. Phys.}, 228(11):4038--4054, 2009.

\bibitem{TadjeranMeerschaert:2007}
C.~Tadjeran and M.~M. Meerschaert.
\newblock A second-order accurate numerical method for the two-dimensional
  fractional diffusion equation.
\newblock {\em J. Comput. Phys.}, 220(2):813--823, 2007.

\bibitem{TadjeranMeerschaertScheffler:2006}
C.~Tadjeran, M.~M. Meerschaert, and H.-P. Scheffler.
\newblock A second-order accurate numerical approximation for the fractional
  diffusion equation.
\newblock {\em J. Comput. Phys.}, 213(1):205--213, 2006.

\bibitem{WangBasu:2012}
H.~Wang and T.~S. Basu.
\newblock A fast finite difference method for two-dimensional space-fractional
  diffusion equations.
\newblock {\em SIAM J. Sci. Comput.}, 34(5):A2444--A2458, 2012.

\bibitem{WangYang:2013}
H.~Wang and D.~Yang.
\newblock Wellposedness of variable-coefficient conservative fractional
  elliptic differential equations.
\newblock {\em SIAM J. Numer. Anal.}, 51(2):1088--–1107, 2013.

\bibitem{ZhouTianDeng:2013}
H.~Zhou, W.~Tian, and W.~Deng.
\newblock Quasi-compact finite difference schemes for space fractional
  diffusion dquations.
\newblock {\em J. Sci. Comput.}, 56(1):45--66, 2013.

\end{thebibliography}

\end{document}